\documentclass[10pt]{amsart}
\usepackage{amssymb}
\usepackage{amsmath}
\usepackage{amscd}
\usepackage{mathrsfs}
\usepackage[all]{xy}
\usepackage[colorlinks=true, linkcolor=red,
urlcolor=blue]{hyperref}
\usepackage{tikz}
\usepackage{pgfplots}
\pgfplotsset{compat = newest}
\usepackage{csquotes}

\numberwithin{equation}{section}

\def\today{\number\day\space\ifcase\month\or
 January\or February\or
   March\or April\or May\or June\or
    July\or August\or September\or
   October\or November\or December\fi\
     \number\year}
\theoremstyle{definition}
\newtheorem{thm}{Theorem}[section]
\newtheorem{lem}[thm]{Lemma}
\newtheorem{prp}[thm]{Proposition}
\newtheorem{dfn}[thm]{Definition}
\newtheorem{cor}[thm]{Corollary}

\newtheorem{rmk}[thm]{Remark}
\newtheorem{ntn}[thm]{Notation}
\newtheorem{exa}[thm]{Example}

\newcommand{\beq}{\begin{equation}}
\newcommand{\eeq}{\end{equation}}
\newcommand{\beqr}{\begin{eqnarray*}}
\newcommand{\eeqr}{\end{eqnarray*}}
\newcommand{\bal}{\begin{align*}}
\newcommand{\eal}{\end{align*}}
\newcommand{\bei}{\begin{itemize}}
\newcommand{\eei}{\end{itemize}}
\newcommand{\limi}[1]{\lim_{{#1} \to \infty}}

\newcommand{\dt}{\delta}
\newcommand{\ep}{\varepsilon}

\newcommand{\ph}{\varphi}

\newcommand{\N}{{\mathbb{Z}}_{> 0}}

\newcommand{\Ntwo}{{\mathbb{Z}}_{\geq 2}}

\renewcommand{\Im}{{\operatorname{Im}}}

\newcommand{\Cu}{{\operatorname{Cu}}}
\newcommand{\Full}{{\operatorname{Full}}}

\newcommand{\Aut}{{\operatorname{Aut}}}

\newcommand{\QT}{{\operatorname{QT}}}
\newcommand{\T}{{\operatorname{T}}}
\newcommand{\W}{{\operatorname{W}}}

\newcommand{\Rrc}{{\operatorname{Rrc}}}
\newcommand{\rc}{{\operatorname{rc}}}

\newcommand{\Idl}{{\operatorname{Idl}}}
\newcommand{\Span}{{\operatorname{Span}}}
\newcommand{\Irc}{{\operatorname{Irc}}}

\newcommand{\Fu}{{\operatorname{Full}}}

\newcommand{\cF}{{\mathcal{F}}}

\newcommand{\cH}{{\mathcal{H}}}

\newcommand{\cK}{{\mathcal{K}}}

\newcommand{\cM}{{\mathcal{M}}}

\newcommand{\cR}{{\mathcal{R}}}

\newcommand{\cZ}{{\mathcal{Z}}}

\newcommand{\ca}{C*-algebra}

\newcommand{\hm}{homomorphism}

\title[Rank Ratio Function]{
The Rank Ratio Function and The  Radius of Comparison}  

\author{M. Ali Asadi-Vasfi}

\curraddr{Department of Mathematics, University of Toronto, Toronto, Ontario, M5S~2E4,  Canada.}

\date{\today}

\subjclass[2010]{Primary 46L05;
 Secondary 06F05; 46L55; 19K14; 46L80.}
\keywords{Cuntz semigroups, C*-algebras, rank ratio function, relative radius of comparison. }
\begin{document}
\maketitle
\begin{abstract}
For any given Cuntz semigroup, 
we introduce a  function associated to it,  called the \enquote{Rank Ratio Function}. This function  ensures  global control 
over the oscillation of the rank of any pair of elements in the Cuntz semigroup and, also, plays a very important role 
 in computing  the relative radii of comparison  of the Cuntz semigroups. 
 Along the way, as we study both
  the C*-algebraic and algebraic  aspects of the relative radius of comparison with a deeper analysis, we 
  introduce another function called the ``Radius of Comparison Function" which  provides a method for
manufacturing non-classifiable C*-algebras with different relative radii of comparison. 
\end{abstract}
%
%
\section{Introduction}
Let $S$ be a Cuntz semigroup. 
We introduce a function $\rho \colon (S, S) \to [0, \infty]$, 
called the \emph{Rank Ratio Function}, and  
establish its key connections to the relative radii of comparison of $S$ and continuity of the rank of a full element in the Cuntz semigroup  of a unital stably finite C*-algebra. 
Let us flesh it out.    
Let us denote by $\cF(S)$ the set of all functionals on $S$. For $w \in S$, we let denote by $\cF_w(S)$ the set all functionals on $S$ which are normalized at $w$  
 (see Definition~\ref{FunctionalDFn}).
For any $x \in S$, the rank of $x$, denoted by 
$\widehat{x}$, is defined as the map $\widehat{x} \colon \cF(S) \to [0, \infty]$ given by $\widehat{x} (\lambda) = \lambda (x)$ for  $\lambda \in \cF(S)$.
For any pair of elements $x, y \in S$, the rank ratio function
 at $(x, y)$ gives us the smallest $r \in (0, \infty)$
 such $\widehat{x} \leq r \cdot \widehat{y}$ (see Definition~\ref{rk_rDef}).   In 
Lemma~\ref{RkRatioProperty}
and Lemma~\ref{IncDesrho},
we study the most important  facts about rank ratio functions, including its homogeneous mixed-monotonic behavior. Let us denote by $\Fu(S)$ the set of all full elements in $S$. 
We prove, more generally,  in Theorem~\ref{LimofRank}
that
if $(x_n)_{n \in \N}$ is an increasing sequence in $S$ with $\sup_n x_n =x$ for $x \in S$. Then:
\begin{enumerate}
\item
The sequence  
$(\rho(x_n, x))_{n=1}^{\infty}$ in an increasing sequence in $\mathbb{R}$, bounded above by 1. So, $\lim_{n \to \infty} \rho(x_n, x)$ exists and is in $\mathbb{R}$.
If further $x_n \in \Full (S)$ for all $n \in \N$ and $\cF_x(S)\neq \emptyset$, then 
$0<\lim_{n \to \infty} \rho(x_n, x)$.
\item 
The sequence  
$(\rho(x, x_n))_{n=1}^{\infty}$ in an decreasing  sequence in $[0, \infty]$. 
If further $x_n \in \Full (S)$ for all $n \in \N$ and $\cF_x(S)\neq \emptyset$, then 
$\lim_{n \to \infty} \rho(x, x_n) \geq 1$. 
\end{enumerate}
We also show that
for any $x, y \in \Fu (S)$, we have 
$\rho(y, x)= 0$ if and only if 
$\cF_{x} (S)= \emptyset$
(see Lemma~\ref{Rkiszero}).

The normalized version of the rank ratio function is given in Definition~\ref{rk_r_0Def}. The idea is to use normalized functionals at a given full element in $S$,  instead of all functionals on $S$.  
Let $z \in \Fu(S)$ with $\cF_z(S) \neq \emptyset$. For for any $x, y \in S$ with $x, y \ll \infty\cdot z$, we denote by $\rho_z (x, y)$ the normalized rank ratio at $(x, y)$. 
We prove in Proposition~\ref{Prp.rhoequiavalent}(\ref{Prp.rhoequiavalent.c}) that  if $A$ is a unital residually stably finite, then, for any pair of full elements $a, b \in \cup_{n=1}^{\infty}M_n (A)_+$, 
\[
\rho ([a], [b]) = \rho_z ([a], [b]) = \sup_{\lambda \in \cF_{[1_A]} (\Cu(A)) } \left( \frac{\lambda ([a])}{\lambda ([b])} \right) < \infty.
\]

The main importance of the rank ratio function is that it can establish a bridge between the relative radii of comparison, which are based on the Cunntz semigroups, and 
the oscillation structure of the rank of operators which are closely connected to their continuity. Let us explain. 
  For any element $w \in \Fu(S)$, 
  we denote by $\rc (S, w)$ the radius of comparison of $S$ relative to $w$ (see Definition~\ref{DefRcCuAlgebraic}).  
  If $x, y \in \Fu(S)$ with $x\leq y$, we know that 
  $\rc (S, y) \leq \rc(S, x)$ (see Proposition~\ref{rcProperties}(\ref{rcProperties.a_1})). 
  One may ask whether there exits a better lower bound for $\rc (S, x)$, together with an appropriate upper bound. 
In theorem~\ref{Rkrc}, we prove that
if $x, y$ are full element  in a $\Cu$-semigroup $S$ with $x\leq y$,   further $\cF_y (S) \neq \emptyset$, and 
 $\rho(y, x)< \infty$, 
then
\[
\frac{1}{\rho(x, y)} \cdot \rc (S, y) \leq \rc(S, x)
\leq 
\rho(y, x) \cdot \rc(S, y). 
\]
See Lemma~\ref{RkRation0} for the case that  $\cF_y (S) = \emptyset$ and $\rho(y, x)=0$. 
More generally, we prove in Theorem~\ref{MainThmForRCLim}
that if $y$ and a sequence $(y_n)_{n=1}^{\infty}$ in $\Full(S)$ with $y_n \leq y_{n +1}\leq y$ for all $n \in \N$,  further $\cF_y (S) \neq \emptyset$, and
$\lim_{n \to \infty} \rho(y, y_n) < \infty$,  
then 
\[
\frac{1}{\lim_{n \to \infty} \rho(y_n, y)} \cdot  \rc (S, y) \leq
 \lim_{n \to \infty} \rc(S, y_n)
\leq 
\lim_{n \to \infty} \rho(y, y_n) \cdot \rc(S, y). 
\]
Therefore, if $\lim_{n \to \infty} \rho(y, y_n) =1$, then
$
 \lim_{n \to \infty} \rc(S, y_n)
= \rc(S, y)$. 
(see Corollary~\ref{rclimrho1})

Additionally, for every
 $\Cu$-semigroup $S$ with $\Full (S)\neq \emptyset$,  
 we define the
\emph{radius of comparison function}, denoted by 
$\rc(S, \cdot)$, from  $\Fu (S)$  to the interval  
$[0, \infty]$, given by $x \mapsto \rc (S, x)$. 
In Proposition~\ref{IrcProperties}, 
we prove some important basic properties of this function. Namely, 
we show that 
the reciprocal of $\rc (S, \cdot)$  is homogeneous  and order preserving. 
This is the reason why we should also consider the reciprocal of $\rc (S, \cdot)$.  
 
As for the range of the radius of comparison function, we show that, for  a unital simple classifiable C*-algebra $A$, 
 the range of $\rc(\Cu (A), \cdot)$ is $\{0\}$
 (see Lemma~\ref{IrcClassi} and Example~\ref{Exa.PurelyInfinite}(\ref{Exa.PurelyInfinite.d})).   
Also,   
this function plays an important role in detecting 
  nonisomorphic, nonclassifiable C*-algebras.
  For nonclassifiable C*-algebra $A$ with stable rank one,  
   our approach leverages
a positive solution to the \emph{rank problem} for C*-algebras with stable rank one \cite{APRT22, Thi20}
and our result here 
to
generate any arbitrary values in the interval $[0, \infty)$ for the relative radii of comparison of $\Cu(A)$ 
 More precisely, 
 we  show in Theorem~\ref{ThmManufactureRc} that
 if $A$ is unital separable residually stably finite C*-algebra $A$ with stable rank one, with no nonzero, finite dimensional quotients, and with $0<\rc (\Cu(A), [1_A])<\infty$, then
the range of $\rc (\Cu (A), \cdot)$ is $[0, \infty)$. 
Applying this and  
 Corollary~5.21 in \cite{ASV23}, we prove in
 Theorem~\ref{NonIsoClassofC*} that, 
for every finite group $G$, there exist
an uncountable family of pairwise nonisomorphic 
unital simple separable Villadsen algebras
 $\{A_{\gamma} 
 \}_{\gamma \in \Lambda}$
 and outter group actions 
$\{\alpha_{\gamma} \colon  G \to A_{\gamma} \}_{{\gamma \in \Lambda}}$,
where $\Lambda$ is an uncountable subset of $(0, \infty)$, 
such that all $A_{\gamma}$ have stable rank one and 
 both the rang of 
 $\rc(\Cu (A_{\gamma}),\cdot)$ and the range 
 $\rc( \Cu( A_{\gamma} \rtimes_{\alpha_{\gamma}} G),  \cdot)$
 are $[0, \infty)$
 for all $\gamma$. 
 Also,  Theorem~\ref{ThmManufactureRc}
 result in a strengthened version of Theorem~3.11 in 
\cite{HP25}  (see Proposition~\ref{IrcNonClassi}). 

Along the way, 
we study some important structural features 
 of
the relative radius of comparison 
in both C*-algebraic and algebraic setting.
In particular,  
 for a unital residually stably finite C*-algebra $A$, we give an appropriate definition of  the radius of comparison of $A$ relative to a full positive element $a$ in $\cup_{n=1}^{\infty} M_n (A)$, denote by $\Rrc(A, a)$. We further prove in Theorem~\ref{CorRrc=rc} that the radius of $\rc (\Cu (A), [a]) = \Rrc (A, a)$ for this class of C*-algebras. 
 The radius of comparison serves as a numerical invariant, extending the concept of the covering dimension to noncommutative. 
This notion originally comes from the Murray-von Neumann comparison theory for factors and,  
in the setting of C*-algebras,  this is just Blackadar’s ``Fundamental Comparability Property" for positive elements (see \cite{B10}).  Roughly speaking, 
for unital stably finite C*-algebra $A$, 
the radius of comparison  of $A$ relative to $1_A$ measures the failure of strict comparison in
the Cuntz semigroup of $A$. 
This was defined in Definition~6.1 in \cite{Tom06} for unital stably finite C*-algebras to 
 construct non-isomorphic simple AH algebras with the same Elliott invariant. 
 While Definition~6.1 in \cite{Tom06} is effective  for simple C*-algebras, 
  it is more appropriate to consider the class of unital residually stably finite C*-algebra when going beyond the setting of nonsimple C*-algebras (see Lemma~\ref{LemFullRrc}).  
  This quantity is often zero  for 
  $\cZ$-stable C*-algebras,  some of which are known as classifiable ones. 
   For
  non-$\cZ$-stable cases, it is strictly  greater than zero and has not been studied extensively. 
Nonetheless, there are some works which help a better understanding such C*-algebras and provide partial results on classification of such C*-algebras and their crossed products (see \cite{ASV23,  AHS25, AE25, H24, Ni22}). 
  Further, the definition of the
radius of comparison was extended  
to a $\Cu$-semigroup $S$ in Definition~ 3.3.2 of \cite{BRTW12}, using  algebraic
properties of $S$.
   The idea is to use
   a full element $x$ in $S$ instead of the Cuntz class of the unit of $A$ and  use the 
   the all the functionals on $S$ instead of normalized quasitraces. 
     This version the radius of comparison 
  is called  ``\emph{the radius of comparison of $S$  relative to $x$}'', or simply \emph{the relative radius of comparison of } $S$ (see Definition~\ref{DefRcCuAlgebraic}). 
  In the setting of C*-algebras, the relative radius of comparison is closely connected to the Corona Factorization Property, introduced  in \cite{KN06}. 
  In fact, if 
  the  relative 
   radius of comparison of the Cuntz semigroup of a C*-algebra is finite, then 
   the C*-algebra has
   the Corona Factorization Property. (See  Theorem~4.2.1 in \cite{BRTW12}.)

As for the connection between rank ration function and continuity of the rank of an element,
we prove in Theorem~\ref{ContRank} that 
if  $A$ is  a unital stably finite C*-algebra and $a\in \cup_{n=1}^{\infty} M_n(A)_+$  is full, then 
$\lim_{n \to \infty} \rho_{[1_A]} \left([a], [(a-\tfrac{1}{n})_+]\right)=1$
if and only if
 $\widehat{[a]}$ is continuous on 
$\cF_{[1_A]} (\Cu (A))$. 
This result has a point of contact with 
oscillation of a positive element in a C*-algebra, which was originally defined in Section A.1 of \cite{ElGONi20}
(see Theorem~\ref{ContRank}).

Finally, for a better understanding of our results, in Section~\ref{SectionExam}, we exhibit  many interesting  examples in both C*-algebra and more algebraic settings that demonstrate how  our theorems can be applied. 
%
\subsection*{Acknowledgments}
The author is indebted to Hannes Thiel for
the several discussions they had over the course of writing this paper. 
The author would like to thank 
George Elliott, Ilan Hirshberg, Feodor Kogan,
Anshu Nirbhay, Zhuang Niu, 
  N. Christopher Phillips, and Andrew S. Toms for their comments. 
The author is also very grateful for the support and hospitality provided
by the Fields Institute, with special thanks to Miriam Schoeman.
\section{Preliminaries }\label{Sec_Prelim}
%
In this  section, we  set up some notation and
collect briefly 
some information on 
the Cuntz semigroups, functionals on them, and their full elements for easy reference.
\begin{ntn}
Throughout, 
if $A$ is a \ca, we write $A_{+}$ for the
set of positive elements of $A$.
 We let $\cK$ denote the algebra of compact operators on a separable Hilbert space $\cH$.
\end{ntn}

We recall the definition of the Cuntz semigroups,  originally from~\cite{Cun78},  and some other basic facts about them. 
\subsection{Cuntz subequivalence}\label{Sec_CuSub}
Let $A$ be a \ca.
\begin{enumerate}
\item\label{Cuntz_def_property_a}
For $a, b \in (A \otimes \cK)_{+}$,
we say that $a$ is {\emph{Cuntz subequivalent to~$b$ in~$A$}},
written $a \precsim b$,
if there exists a sequence $(w_n)_{n = 1}^{\infty}$ in $A \otimes \cK$
such that
$
\limi{n} w_n b w_n^* = a$.
We say that $a$ and $b$ are {\emph{Cuntz equivalent in~$A$}},
written $a \sim_{A} b$,
if $a \precsim b$ and $b \precsim a$.
This relation is an equivalence relation,
and we write $ [a]_A$, or simply $[a]$ when unambiguous, for the equivalence class of~$a$.
We define the positively ordered monoid
$
\Cu (A) = (A \otimes \cK)_{+} / \sim$,
together with the commutative semigroup operation
$[a]_A +  [b]_A
 = [a \oplus b]_A$ coming from 
  an isomorphism $M_2(\cK) \to \cK$
and the partial order
$[a]_A \leq [b]_A$
if $a \precsim b$.
We write $0$ for~$[0]_A$.
\item\label{Cuntz_def_property_e}
Let $A$ and $B$ be C*-algebras
and let $\ph \colon A \to B$ be a \hm.
We use the same letter for the induced maps
$M_n (A) \to M_n (B)$
for $n \in \N$ and
$A \otimes \cK \to B \otimes \cK$.
We define
 $\Cu (\ph) \colon \Cu (A) \to \Cu (B)$
by $[a]_A \mapsto [\ph (a)]_B$
for  $a \in (A \otimes \cK)_{+}$.
\item
Considering  the usual identifications 
$A \subset A \otimes M_n  \subset A \otimes \cK$,
the \emph{pre-completed Cuntz semigroup of} $A$, denoted by $\W(A)$, is 
 $
 \W(A)= \cup_{n=1}^{\infty} M_n (A)_{+} / \sim
$.
\end{enumerate}
\begin{dfn}
Let $A$ be a \ca,
let $a \in A_{+}$,
and let $\ep \geq  0$.
Let $f \colon [0, \infty) \to
[0, \infty)$  be the function
$f (t) = \max (0, \, t - \ep) = (t - \ep)_{+}$.
Then, by functional calculus, define $(a - \ep)_{+} = f (a)$.
\end{dfn}
Part (\ref{PhiB.Lem.18.4.11}) of the following lemma is called 
\emph{Rørdam’s lemma} and is Proposition 2.4  of~\cite{Ror92}.
 Part (\ref{Item.largesub.lem1.7}) is Lemma~1.7 of \cite{Ph14}.
\begin{lem} 
\label{PhiB.Lem.18.4}
 Let $A$ be a \ca .
 \begin{enumerate}
 \item\label{PhiB.Lem.18.4.11}
 Let $a, b \in A_{+}$.
 Then the following relations are equivalent:
 \begin{enumerate}
 \item\label{PhiB.Lem.18.4.11.a}
 $a \precsim_A b$.
 \item\label{PhiB.Lem.18.4.11.b}
 $(a - \ep)_{+} \precsim_A b$ for all $\ep > 0$.
 \item\label{PhiB.Lem.18.4.11.c}
 For every $\ep > 0$ there is $\dt > 0$ such that
 $(a - \ep)_{+} \precsim_A (b - \dt)_{+}$.
 \end{enumerate}
\item\label{Item.largesub.lem1.7}
Let $\ep > 0$. Let $a, b \in A$ satisfy $0 \leq a \leq b$. 
Then 
$(a - \ep)_+ \precsim_A (b - \ep)_+$.
 \end{enumerate}
 \end{lem}

\subsection{The category $\Cu$}
 The Cuntz semigroup is a very powerful and useful tools for capturing the structure of C*-algebra. It was first applied in \cite{Cun78, BH82} to  study extensively the quasitarces on a C*-algebra and its association to the functionals on a Cuntz semigroup. Over the last two decades, both the Cuntz semigroup and radius of comparison have played key roles in formulating important regularity properties which contribute significantly to completion of the Elliott classification program \cite{CE08, ET08, Tom08} for simple amenable C*-algebras. 
 
The following definition is from \cite{CowEllIva08CuInv}.  (See Theorem~1 on page 163 of \cite{CowEllIva08CuInv}.)
\begin{dfn}
Let $(P, \leq)$ be a partially ordered set. 
Let $x, y \in P$. We say that $x$ is \emph{compactly contained} in $y$, denoted by $x \ll y$, if for any increasing sequence 
$(y_n)_{n \in \N}$ in $P$ with supremum $y$, we have
$x \leq y_{n_0}$ for some $n_0 \in \N$.
A \emph{$\Cu$-semigroup} is a positively ordered abelian monoid~$S$ with the following properties:
\begin{enumerate}
\item[(O1)]
Every increasing sequence in $S$ has a supremum.
\item[(O2)]
For every element $x$ in $S$, there exists a sequence $(x_n)_{n \in \N}$ in $S$ such that $x_0 \ll x_1 \ll x_2 \ll \cdots$ and such that $x = \sup_n x_n$.  
\item[(O3)] 
For every $x', x, y', y  \in S$ satisfying $x' \ll x$ and $y' \ll y$, we have $x' + y' \ll x + y$.
\item[(O4)]
For every increasing sequences $(x_n)_{n \in \N}$ and $(y_n)_{n \in \N}$ in $S$, we have 
\[
\sup_n (x_n+y_n) = \sup_nx_n + \sup_ny_n.
\]
\end{enumerate} 
\end{dfn}
It is shown in \cite{CowEllIva08CuInv} that the Cuntz semigroup $\Cu(A)$ of any C*-algebra is a $\Cu$-semigroup. 
Also, we refer to \cite{APT18} for detailed results on the 
regularity properties of Cuntz semigroups.

For $x \in S$, we denote by $\infty \cdot x$ the supremum $\sup_{n\in \N} n \cdot x$. 
Also, we denote by $\Idl(x)$ the ideal generated by $x$, that is
$\Idl(x) = \{y \in S \colon y \leq \infty \cdot x\}$.
\begin{dfn}\label{AbstractDfnFull}
Let $S$ be a $\Cu$-semigroup and let $x \in S$. We say that $x$ is full if
$\infty \cdot x$ is the largest element of $S$, or equivalently, $\Idl(x)=S$.
We let $\Fu (S)$ denote the set of full elements of $S$. 
\end{dfn}
For given $x, y$ in a $\Cu$-semigroup $S$, once can easily check that
$\Idl(x) \subseteq \Idl(y)$ if and only if $\infty \cdot x \leq \infty \cdot y$. 
In general, not all $\Cu$-semigroups  have a largest element. But for simple or countably-based 
$\Cu$-semigroups, a largest element does exist. See the discussion before Lemma~5.2.3 in \cite{APT18} for more details. 
%
\subsection{Functionals and extended
2-quasitraces}
\begin{dfn}\label{FunctionalDFn}
Let $S$ be a $\Cu$-semigroup. 
A \emph{functional}
on $S$ is a map 
$\lambda \colon S \to [0, \infty]$ which is
 additive, order preserving, satisfies $\lambda(0) = 0$, and preserves suprema of
increasing sequences in $S$.
\end{dfn}
We use $\cF(S)$ to denote the functionals on $S$ and,  for $z\in S$, we use $\cF_z(S)$ to denote 
the set of all $\lambda \in \cF(S)$ with 
$\lambda(z)=1$. 
We recall the following trivial
functional on a $\Cu$-semigroup $S$: 
\[
\lambda_{\infty} (x) = 
\begin{cases}
             0,  & \text{if } x=0 \\
             \infty,  & \text{if }  x\neq  0.
       \end{cases}
\]

\begin{dfn}
For a C*-algebra $A$,  a \emph{lower semicontinuous extended 2-quasitrace} on $A$ is a lower semicontinuous map
$\tau \colon (A \otimes \cK)_+ \to [0, \infty]$ which satisfies the following conditions hold:
\begin{enumerate}
\item
$\tau(a a^*) = \tau(a^* a)$ for all $a \in A \otimes \cK$. 
\item
$\tau (0) = 0$ and 
$\tau(s\cdot a) = s\cdot \tau(a)$ for all $a \in (A \otimes \cK)_+$ and all $s \in (0, \infty)$. 
\item 
$\tau(a+b) = \tau(a)+\tau(b)$ for all
 $a, b\in (A \otimes \cK)_+$ with  $ab=ba$. 
\end{enumerate}
\end{dfn}
For a C*-algebra $A$, we let $\QT_2 (A)$ denote
 the set of the lower semicontinuous extended 2-quasitraces on $A$ and 
 if $A$ is unital, we denote by
 $\QT^{1}_{2} (A)$
 the set of normalised 2-quasitraces on $A$
and denote by
$\T^{1} (A)$ the set of normalized traces on $A$. 
 It follows from Corollary 6.10 of \cite{Thi17} that extended quasitraces are order-preserving, i.e., if $a, b \in (A\otimes \cK)_+$ with 
$a \leq b$, then $\tau (a) \leq \tau (b)$ for $\tau \in \QT_2 (A)$.

It is shown in \cite{BH82} any unital stably finite C*-algebra admits a normalized 2-quasitrace.

For any $\tau \in \QT_2 (A)$, we define a map
$d_{\tau} \colon \Cu(A) \to [0, \infty]$ by the following formula:
\[
d_{\tau} (a)
 = \lim_{ n \to \infty} 
\tau \big(a^{ \tfrac{1}{n} }\big).
\] 
\begin{rmk}\label{rmkHandelElliott}
For a unital C*-algebra $A$,
it was shown by Blackadar and Handelman in
 \cite{BH82} that the association $\tau \mapsto d_\tau$ defines a bijection between $\QT^{1}_2(A)$
and $\cF_{[1_A]}(\Cu(A))$. 
In a more general setting, 
it is shown in Proposition 4.2 of \cite{ERS11} that
for a C*-algebra $A$ (not necessarily unital),  
 the association $\tau \mapsto d_\tau$ defines a homeomorphism  between $\QT_2(A)$
and $\cF (\Cu(A))$. 
\end{rmk}
\subsection{Full Elements of C*-algebras}
\begin{dfn}
Let  $A$ be  a \ca{} and let $a \in A$.
 We say $a$ is \emph{full} (in $A$) if the closed two-sided
ideal of $A$ generated by $a$ is all of $A$ (i.e., $\overline{\Span(AaA)}=A$). 
The set of all full elements in $A$ is denoted by $\Fu (A)$. Similarly, for a C*-subalgebra $B$ of $A$, we say that
$B$ is full if $\overline{\Span(ABA)}=A$. 
\end{dfn}

If $a$ is full in $A$, then $a$ also full in $M_n (A)$ for all $n \in \N$ and in $A \otimes \cK$.  
For a C*-algebra $A$ and $a \in (A \otimes \cK)_+$, it follows from Proposition~5.1.10 of \cite{APT18} that
 $a$ is full in $A \otimes \cK$
if and only if $[ a ]_A$ is full in $\Cu(A)$
in the sense of Definition~\ref{AbstractDfnFull}. 

\begin{lem}\label{FuSeq.El.Lem}
Let $A$ be a unital C*-algebra and let $a \in A_+$ be full. Let $(a_n)_{n=1}^{\infty}$ be a sequence in $A_+$ with $\lim_{n \to \infty} a_n =a$. 
Then there exists $n_0 \in \N$ such that  $a_m$ is full
for all $m \geq n_0$. 
\end{lem}
\begin{proof}
Since $a$ is full, there exist $t \in \N$ and $b_1, b_2, \ldots, b_t$ in $A$ such that 
$1_A =\sum_{j=1}^{t} b_j a b^*_j$. 
Let $\ep \in (0, 1)$. Set 
$M= \max_{1\leq j \leq t} (\|b_j \|+1)$ and 
$\dt = \min (\tfrac{\ep}{M^2}, 1)$. 
Since $\lim_{n \to \infty} a_n =a$, we choose
$N\in \N$ such that
$\|a - a_n \| < \dt$ for all $n \geq N$. 
Now, for all $n \geq N$, we have 
\begin{align*}
\left \|1_A - \sum_{j=1}^{t} b_j a_n b^*_j \right\|
&=
\left\|\sum_{j=1}^{t} b_j a b^*_j  - \sum_{j=1}^{t} b_j a_n b^*_j \right\|
\\& \leq 
\sum_{j=1}^{t} \| b_j\| \cdot \| a - a_n\| \cdot \| b^*_j\| < \ep < 1.
\end{align*}
This implies that $\sum_{j=1}^{t} b_j a_n b^*_j \in Aa_n A$ is invertible for $n \geq N$, and therefore, $\overline{\Span (Aa_n A)} =A$ for all $n \geq N$.  This complete the proof. 
\end{proof}
 In the following lemma, we characterize all full elements in $C(X) \otimes A$ for a simple
C*-algebra $A$.
\begin{lem}\label{FullElements_Lem}
Let $X$ be a compact  metric space and let $A$ be a  simple \ca. Then
\[
\Fu \left(C(X, A)\right)
=
 \Big\{ b \in C(X, A) \colon b(x)\neq 0 \mbox{ for all } x\in X \Big\}.
\] 
\end{lem}

\begin{proof}
Set 
$B= \left\{ b \in C(X, A) \colon b(x)\neq 0 \mbox{ for all } x\in X \right\}$. 
First, we prove $B \subseteq \Fu\left(C(X, A)\right)$.
Let $a \in C(X, A)$, let $b \in B$, and let $\ep>0$.
It suffices to find an element 
$v \in \Span ( C(X, A) b C(X, A) )$ such that 
$\| a - v \|< \ep$.

For $x\in X$, 
since $A$ is simple and $b (x)\neq 0$, we have $A = \overline{\Span (A b(x) A)}$. Then
 $a(x) \in \overline{ \Span (A b(x) A)}$. 
Therefore, there exist $m_x \in \N$ and  $v^{(k)}_{x}, w^{(k)}_x \in A$ for $k= 1, \ldots, m_x$ such that
\begin{equation}
\label{EQ.333tqo}
 \left\| \sum_{k=1}^{m_x} v^{(k)}_{x} b(x) w^{(k)}_{x} - a(x)\right\|< \ep.
\end{equation}
Since $a, b \in C(X, A)$, the   map $\xi_x \colon X \to [0, \infty)$, 
given by 
$
z \mapsto \left\| \sum_{k=1}^{m_x} v^{(k)}_{x} b(z) w^{(k)}_{x} - a(z) \right\|,
$
 is continuous.
Using this and  $\xi_x (x)<\ep$, there exists an open neighborhood $U_x$ of $x$ such that, for all $z\in U_x$,  
 \begin{equation}\label{Eq191012t2qo}
 \left\|\sum_{k=1}^{m_x} v^{(k)}_{x} b(z) w^{(k)}_{x} - a(z) \right\|< \ep.
\end{equation}
 Since $X$ is compact and $X=\bigcup_{x \in X} U_x$, we can choose
$x_1, \ldots, x_n \in X$ with $X=\bigcup_{j=1}^{n}U_{x_j}$ for some $n\in \N$.
Choose continuous functions $h_j \colon X \to [0, 1]$ for $j=1, \ldots, n$ 
which form a partition of unity on $X$ and such that 
\begin{equation}\label{EQ1.191025}
\mathrm{supp}(h_j) \subseteq U_{x_j}
\quad
\text{ and } 
\quad
\sum_{j=1}^{n} h_j =1.
\end{equation}

Set  
$v = \sum_{j=1}^{n} \sum_{k=1}^{m_{x_j}} h_j v^{(k)}_{x_j} b w^{(k)}_{x_j}.
$
Clearly  $v \in \Span ( C(X, A) b C(X, A) )$. 
Now we claim 
$\| a - v \| < \ep$.
To prove the claim, let $z \in X$. 
Define  
$\Lambda_z = \big\{ j \in \{1, \ldots, n\} \colon z \in U_{x_j} \big\}$.
Now, we estimate
\begin{align*}
\left\| a(z) - v(z) \right\|
&
\underset{\text{(\ref{EQ1.191025})}}{=}
\left\| 
\left( \sum_{j=1}^{n} h_j (z)\right) a(z) 
- 
\sum_{j=1}^{n} \sum_{k=1}^{m_{x_j}} h_j (z) v^{(k)}_{x_j} b(z) w^{(k)}_{x_j}
 \right\|
\\
&\leq
 \sum_{j=1}^{n} h_j (z) \left\| a(z) - \sum_{k=1}^{m_{x_j}} v^{(k)}_{x_j} b(z) w^{(k)}_{x_j} \right\|
\\
&
\overset{\text{(\ref{EQ1.191025})}}{\underset{\text{(\ref{Eq191012t2qo})}}{<}}
\sum_{j \in \Lambda_z} h_j (z) \ep
\leq
\sum_{j=1}^{n} h_j (z) \ep 
\underset{\text{(\ref{EQ1.191025})}}{=} \ep.
\end{align*}

To prove $\Fu \left(C(X, A)\right) \subseteq B$, let $f \in C(X, A)$ satisfy
$\overline{ \Span ( C(X, A) f C(x, A) )} = C(X, A)$. 
Assume that there exists $x_0 \in X$ such that $f(x_0)=0$.
Choose $c \in A$ with $\|c\|=1$ and $\ep\in (0, 1)$. Then define 
$g \colon X \to A$ by $g(x)=c$ for all $x\in X$.
Therefore, there exist $m \in \N$ and  $s^{(k)}, t^{(k)} \in C(X, A)$ for $k= 1, \ldots, m$ such that
$
 \left\| \sum_{k=1}^{m} s^{(k)} f t^{(k)} - g\right\|< \ep.
$
Using this at the second step, we get
\[
1=\left\| \sum_{k=1}^{m} s^{(k)}(x_0) f(x_0) t^{(k)}(x_0) - g(x_0)\right\|< \ep<1.
\]
This is a contradiction. 
\end{proof}
We conclude this section by the following lemmas about full elements. 
\begin{lem}\label{InfofFullElem}
Let $A$ be a unital stably finite C*-algebra and let $a \in \cup_{n=1}^{\infty} M_n (A)_+$ be full. 
Then:
\begin{enumerate}
\item\label{InfofFullElem.a}
$\tau(a) > 0$ for all $\tau \in \QT^1_2(A)$. 
\item \label{InfofFullElem.b}
$\inf_{\tau \in \QT^1_2(A)} \tau (a) >0$.
\item \label{InfofFullElem.c}
$\inf_{\tau \in \QT^1_2(A)} d_{\tau} (a) >0$. 
\end{enumerate}
\end{lem}
\begin{proof}
To prove (\ref{InfofFullElem.a}), 
assume $\tau (a)=0$ for some $\tau \in \QT^1_2(A)$. So, $d_\tau (a)=0$.
Since $[1_A] \ll \infty \cdot [a]$. Choose $k \in \N$ such that 
$[1_A] \leq k \cdot [a]$. So, 
$\tau (1_A) = d_\tau (1_A) \leq d_\tau (a)=0$.
 This is a contradiction.   

Part (\ref{InfofFullElem.b}) follows from the continuity of the map $g \colon \QT^1_2(A) \to 
[0, \infty)$, given by $g(\tau) = \tau (a)$, and (\ref{InfofFullElem.a}). 

We prove (\ref{InfofFullElem.c}). We know that 
$\tau \left(\tfrac{1}{\| a \|} \cdot a \right) \leq 
d_{\tau} \left(\tfrac{1}{\| a \|} \cdot a \right)$.
We use this and (\ref{InfofFullElem.b}) to get 
$0 < \inf_{\tau \in \QT^1_2(A)} \tau (a) \leq
\inf_{\tau \in \QT^1_2(A)} d_{\tau} (a)$. This complete the proof. 
\end{proof}

Recall that a unital C*-algebra is residually stably finite if  all its quotients are stably finite. 
The following lemma provides a useful method to verify a full element in unital residually stably finite C*-algebras. 
\begin{lem}\label{ResiduallyAndFullEquivalent}
Let $A$ be a unital residually stably finite C*-algebra
and let $a\in \cup_{n=1}^{\infty} M_n (A)_+$. Then 
$a$ is full if and only if $\inf_{\lambda \in \cF_{[1_A]}(\Cu(A))} \lambda ([a])>0$. 
\end{lem}
\begin{proof}
The forward implication from Lemma~\ref{InfofFullElem}
and Remark~\ref{rmkHandelElliott}.
So, we prove the backward implication
Set $I =\overline{\Span (A\otimes \cK) a (A\otimes \cK)}$. and assume that $a$ is not full. So, $0\neq I \subset A \otimes \cK$. 
 Since $A$ is residually stably  finite, it follows that $A \otimes \cK / I$ is finite, so we can
define $\bar{\lambda} \in \cF \Big( \Cu \big( (A \otimes \cK) / I\big) \Big)$ which is nonzero and $\bar{\lambda} (1_A + I)=1$.
Then
$\bar{\lambda}$ induces a functional on 
$\lambda$ on $\Cu (A)$ such that $\lambda ([1_A])= 1$ and $\lambda([a])=0$. 
This is a contradiction. 
\end{proof}
\section{The relative radius of comparison revisited}
\label{RCrevisited}
In this section,  we give an appropriate definition of the the radius of comparison of $A$ relative to a full element $a$ in $\cup_{n=1}^{\infty} M_n (A)$, denote by $\Rrc(A, a)$ for a unital residually stably finite C*-algebra $A$. This is a generalization of the conventional radius of comparison, introduced by Andew S. Toms in \cite{Tom06}. 
 We further show that, for this class of C*-algebras, 
$\Rrc(A, a)$ for a full element $a$ in $\cup_{n=1}^{\infty} M_n (A)$ is equal to the radius of comparison of $\Cu(A)$
relative to $[a]$, i.e., $\rc (\Cu (A), [a])$ which is introduced in Definition~ 3.3.2 of \cite{BRTW12} and is denoted by $r_{A, a}$ there. 
\begin{dfn}
 Let $S$ be a $\Cu$-semigroup and let $a \in S$.  The rank of 
 $a$, denoted by $\widehat{a}$, is defined as the map 
 $\widehat{a}\colon \cF(S) \to [0, \infty]$ given by 
 $\widehat{a}=\lambda (a)$. 
\end{dfn}

If we restrict the domain of rank of a given element in a $\Cu$-semigroup $S$, we explicitly mention it; when there is no risk of confusion, the domain is  $\cF(S)$ by default. 
 
 The following definition is contained in Definition~3.3.2 of \cite{BRTW12}. 
\begin{dfn}\label{DefRcCuAlgebraic}
Let $S$ be a $\Cu$-semigroup   and let $w \in \Fu (S)$.
\begin{enumerate}
\item
Let $r\in (0, \infty)$. We say that 
$S$  has \emph{$r$-comparison relative to $w$}, if whenever $x, y \in S$
satisfy 
$
\widehat{ x } + r \cdot \widehat{w}
 \leq
\widehat{y}$, 
then $x \leq y$. 
\item
 The \emph{radius of comparison of $S$ relative to
$w$}, denoted by
 $\rc (S, w)$, is 
 \[
 \rc (S, w)= \inf \left\{r \in (0, \infty) \colon 
S~ \mbox{has $r$-comparison relative to $w$ } 
\right\}.
 \]
 If there is no such $r$, then $ \rc (S, w)=\infty$.
\end{enumerate}
\end{dfn}
See Definition~2.20 in \cite{AE25} for 
the C*-algebraic of the above definition. 

We know that if $x\leq y$, then $\Idl (x) \subseteq \Idl (y)$. But, the converse is not true in general. 
In the above definition, one may ask if $x, y \in S$ satisfy $
\widehat{ x } + r \cdot \widehat{w}
 \leq
\widehat{y}$, can we even  compare $\Idl (x)$ and 
 $\Idl (y)$, let alone compare $x$ and $y$?   
We answer this question in the following lemma. 
 \begin{lem}
 \label{LemFuRCDf}
 Let $S$ be a $\Cu$-semigroup and let $e \in S$ be  full.
Let $r\in (0, \infty)$. Let $x, y \in S$
satisfy 
$\widehat{ x } + r \cdot \widehat{e}
 \leq
\widehat{y}$, then $y \in \Fu(S)$. 
 \end{lem}
 \begin{proof}
 We may assume that $S$ is a nonzero $\Cu$-semigroup. 
 If
  $y=0$, then $\lambda(e)=0$ for all $\lambda \in \cF(S)$.
  This contradicts the fact that $\lambda_{\infty} (e)=\infty$.
 Now, 
 set $I= \Idl (y)$.
 Define $\lambda \colon S \to [0, \infty]$ by
 \[
\lambda_I (s) =
\begin{cases} 
0, & \text{if } s \in I, \\
\infty, & \mathrm{o.w}.
\end{cases}
\]
It easy to check  that $\lambda_I \in \cF(S)$ and 
$\lambda_I (y)=0$. Using this and the fact that 
$\lambda_I ( x ) + r \cdot \lambda_I (e)
 \leq
\lambda_I ( y )$, we get $\lambda_I (e)=0$, and therefore $e \in I$. This implies that
$\infty \cdot e \leq \infty \cdot y$. Since $e \in \Fu (S)$, it follows that $\infty \cdot e = \infty \cdot y$
and, therefore, $y \in \Fu (S)$, as desired. 
 \end{proof}

Let 
$\Im$ be the operation that assigns to any function its image. 
The following proposition provides some important properties of  the relative radius of comparison in algebraic terms. 
\begin{prp}\label{rcProperties}
 Let $S$ be a $\Cu$-semigroup. Let $x, y \in \Full(S)$ and let $\eta \in (0, \infty)$. Then: 
 \begin{enumerate}
 \item \label{rcProperties.a_0}
 If $\widehat{x} \leq \eta \cdot \widehat{y}$,
 then $\frac{1}{\eta}\cdot \rc (S, y) \leq \rc (S, x)$. 
 \item \label{rcProperties.a_1}
 If $x \leq  y$,
 then  $\rc (S, y) \leq \rc (S, x)$. 
 \item\label{rcProperties.b}
  If $\widehat{x} = \eta \cdot \widehat{y}$,
 then $\frac{1}{\eta}\cdot \rc (S, y) = \rc (S, x)$. 
 \item\label{rcProperties.c_0}
 If $\rc (S, x) < \infty$
and $\Im (\widehat{y}) =\{0,  \infty\}$,  then $\rc (S, y) =0$.
 \item\label{rcProperties.c}
 If $\rc (S, x) < \infty$, then $\rc (S, \infty \cdot x) =0$.
\item \label{rcProperties.d}
For every $n \in \N$, we have 
$\rc (S, n\cdot x)= \frac{1}{n} \cdot \rc(S, x)$.
\item\label{rcProperties.f}
$\rc (S, x+y) \leq \min \big( \rc (S, x), \rc (S, y)\big)$. 
 \end{enumerate}
\end{prp}
\begin{proof}
We prove (\ref{rcProperties.a_0}). 
If $\rc(S, x)= \infty$, then (\ref{rcProperties.a_0}) is trivially true. 
So, let $\rc(S, x) < \infty$. 
 Let $r \in (0, \infty)$. It suffices to show that if
 $S$ has $r$-comparison relative to $x$, then $S$ has $\eta r$-comparison relative to $y$. 
 So, suppose that $S$ has $r$-comparison relative to $x$.
 Let $w, z \in S$ satisfy 
$
\widehat{w}  + (\eta r) \cdot \widehat{y}
 \leq
\widehat{z}$. 
 Using this and 
$\widehat{x} \leq \eta \cdot \widehat{y}$, we get
  $
\widehat{w}  +  r \cdot \widehat{x}
 \leq
 \widehat{z}$.  
 Since $S$ has $r$-comparison relative to $x$, then 
 $w\leq z$. Thus, $S$ has $\eta r$-comparison relative to $y$. 
 
We prove  (\ref{rcProperties.a_1}). 
Since $x \leq y$, it follows from  that $\lambda(x) \leq \lambda (y)$
 for all $\lambda \in \cF(S)$. Now, we use this and  (\ref{rcProperties.a_0}) to get (\ref{rcProperties.a_1}).  
 
 Part (\ref{rcProperties.b}) follows immediately from (\ref{rcProperties.a_0}). 

We prove (\ref{rcProperties.c_0}).
Let $r \in (0, \infty)$. Suppose that $S$ 
has $r$-comparison relative to $x$ (there exists such $r$ because $\rc (S, x)< \infty$). 
Let $n \in \N$. 
Let $a, b \in S$ satisfy 
 \begin{equation}\label{Eq11.2025.03.08}
\widehat{a} + (\tfrac{r}{n}) \cdot \widehat{y}
 \leq
\widehat{ b}.  
\end{equation}
Sinc  $\Im ( \widehat {y} ) = \{0, \infty\}$, we have $r \cdot \lambda (x) \leq (\tfrac{r}{n}) \cdot \lambda (y)$ for all 
$\lambda \in \cF(S)$ (if $\lambda (y)=0$, then $\lambda \equiv 0$). Now, using this and (\ref{Eq11.2025.03.08}), we get
$
 \widehat{a}  + r \cdot \widehat{ x }
 \leq
\widehat{b}$.  
 Since $S$ has $r$-comparison relative to $x$, it follows that 
 $x\leq y$. This implies that 
 $\rc (S, y) \leq \frac{r}{n}$.
 Take infimum over all such $r$ to get
 $\rc (S, y) \leq \frac{1}{n} \cdot \rc (S, x)$. Using $\rc(S, x)<\infty$, we get 
$\rc (S, y)=0$.

 Part (\ref{rcProperties.c})
 follows from (\ref{rcProperties.c_0}) and the fact that
 \begin{equation*}
\lambda (\infty \cdot x) = \lambda \left(\sup_{n \in \N} n\cdot x \right) = \sup_{n \in \N} (n \cdot \lambda (x) )=
\begin{cases} 
0, & \text{if } \lambda(x) = 0, \\
\infty, & \text{if } \lambda (x) \neq 0.
\end{cases}
\end{equation*}
 
 Part (\ref{rcProperties.d}) follows immediately from (\ref{rcProperties.b}) and the fact that $\lambda (n\cdot x)= n \cdot\lambda (x)$
 for all $\lambda \in \cF(S)$ and all $n\in \N$. 
 
 We prove (\ref{rcProperties.f}).
 Since $x \leq x+ y$ and $y \leq x+ y$, it follows from (\ref{rcProperties.a_1}) that 
 $\rc(S, x+y) \leq \rc (S, x)$ and 
  $\rc(S, x+y) \leq \rc (S, y)$. So, (\ref{rcProperties.f}) follows. 
\end{proof}
Recall that 
a C*-algebra $A$ is said to
be \emph{purely infinite} if $A$ has no non-zero abelian quotients and if for each pair
of elements $a, b \in A_+$ such that $a \in \overline{\Span(AbA)}$, we have $a \precsim b$. 
In the following lemma, we show that
for this class of C*-algebras, the relative radius of comparison is zero. This is in fact one of the advantages of working with Definition~\ref{DefRcCuAlgebraic} for this class of C*-algebras. 
Because if we were allowed to use 
Definition~\ref{RrcDf} for unital purely infinite C*-algebras, we would obtain $\infty$ for the value of the relative radius of comparison.

In the following lemma, we explicitly show that
the relative radii of comparison of purely infinite C*-algebras (not necessarily simple) are zero, using Definition~\ref{DefRcCuAlgebraic}. 

\begin{lem}
\label{RcPurelyInfinite}
Let $A$ be a purely infinite C*-algebra.
and let $a \in (A\otimes \cK)_+$ be full. 
Then $\rc (\Cu(A), [a])=0$. 
\end{lem}
\begin{proof}
Since $A$ is purely infinite, it follows from Proposition~3.5 of \cite{KR02} that $A \otimes \cK$ and $M_n(A)$ for all $n\in \N$ are also purely infinite. 
Let $r>0$ and let 
 $x, y \in (A \otimes \cK)_+$
satisfy 
$
\widehat{ [x]}  + r \cdot\widehat{[a]}
 \leq
\widehat{ [y]}$. 
We show that $x \precsim y$. 
We may assume that 
$A$ is nonzero C*-algebra and
$y\neq 0$. 
(If $y=0$, then,
$\widehat{[a]}=0$ which contradicts 
$\lambda_\infty ([a])=\infty$). 
Using Lemma~\ref{LemFuRCDf}, we get that $y$ is full, so $x \in \overline{\Span ((A\otimes \cK) y(A\otimes \cK))}$. Now, it follows from the definition that $x \precsim y$. Therefore, $\rc (\Cu(A), [a])=0$. This completes the proof.
\end{proof}

The following definition is a relative version of the conventional radius of comparison introduced by Andrew Toms in \cite{Tom06}. Here we consider the class of unital residually stably finite C*-algebra instead of unital stably finite C*-algebra
  while defining  $\Rrc(A, a)$. It is mainly because of the fact that Definition~6.1 in \cite{Tom06} was intended to be used for simple unital stably finite C*-algebras. Beyond simple C*-algebras, 
there is a technical issue in the Definition~6.1 of \cite{Tom06} (see Lemma~\ref{LemFullRrc} for more details). 
Nonetheless, if we stick with the class of unital residually 
stably finite C*-algebras, both Definition~6.1 of \cite{Tom06} and Definition~\ref{RrcDf} are perfectly suitable. 
\begin{dfn}\label{RrcDf}
Let $A$ be a unital residually stably finite C*-algebra and let $c \in \cup_{n=1}^{\infty} M_n (A)_+$ be full. 
\begin{enumerate}
\item \label{RrcDf.1}
Let $r\in (0, \infty)$. We say that 
$A$  has \emph{$r$-comparison relative to $c$}, if whenever $x, y \in \cup_{n=1}^{\infty} M_n (A)_+$
satisfy 
$
\widehat{[x]}  + r 
 <
\widehat{[y]}
$
 on  $\cF_{[c]} (\Cu(A))$,
then $x \precsim_A y$. 
\item\label{RrcDf.2}
 The \emph{radius of comparison of $A$ relative to
$a$}, denoted by
 $\Rrc (A, c )$, is 
 \[
 \Rrc (A, c)= \inf \left\{r \in (0, \infty) \colon 
A~ \mbox{has $r$-comparison relative to $c$ } 
\right\}.
 \]
 If there is no such $r$, then $ \Rrc (A, c)=\infty$.
\end{enumerate}
\end{dfn}
The letter R used in the notation $\Rrc(A, c)$ stands for ``relative". Also, 
if we take $c=1_A$, then $\Rrc(A, 1_A)$ is equal to
$\rc (A)$, as defined  in Definition~6.1 in \cite{Tom06}. 
\begin{rmk} 
 The relative radius of comparison does not change under transition to
a matrix algebra over $A$ and only depends on the the given full element, i.e., $\Rrc (M_n(A), c) = \Rrc (A, c)$
for $n \in \N$ and a given full positive element $c$ in 
$\cup_{n=1}^{\infty} M_n (A)_+$. We caution the reader not to confuse this with that fact that $\rc (M_n (A))= \frac{1}{n}\cdot \rc(A)$ as in Proposition~6.2(ii) of \cite{Tom06}. Actually, $\rc (M_n(A))$ in  our setting is 
$\Rrc \Big(M_n (A), \ \bigoplus_{j=1}^n 1_A \Big)$ which is equal to
$\frac{1}{n} \cdot \Rrc (M_n (A), 1_A)$ by Proposition~\ref{rcProperties}(\ref{rcProperties.d}).
\end{rmk}

The following lemma is used in the proof  of Lemma~\ref{LemFullRrc} and Lemma~\ref{PropR1234}. 

\begin{lem}\label{LemcFbigection}
Let $A$ be a unital stably finite C*-algebra and let $a \in \cup_{n=1}^{\infty} M_n (A)_+$ be full.
Let $\lambda \in \cF(S)$.  
\begin{enumerate}
\item\label{LemcFbigection.a}
$0<\lambda ([a]) <\infty$ in and only if $0<\lambda ([1_A])<\infty$. 
\item 
\label{LemcFbigection.b}
The map $\varphi \colon \cF_{[1_A]}(\Cu(A)) \to \cF_{[a]}(\Cu(A))$, given by 
$\lambda \mapsto \tfrac{1}{\lambda(a)} \cdot \lambda$, is bijective. 
\end{enumerate}
\end{lem}
\begin{proof}
It is easy to check. 
\end{proof}
Consider $x, y$ in Definition~\ref{RrcDf}(\ref{RrcDf.1}). 
It is natural to ask whether 
the ideal generated by  $x$ and the ideal generated by $y$
can be compared. 
If we cannot do this, then comparing $x$ and $y$ in the Cuntz semigroup would not make sense.  
The answer of the question is ``Yes" in the setting of unital residually stably finite C*-algebra and we show it in the following lemma.
\begin{lem}\label{LemFullRrc}
Let $A$ be a unital residually stably finite C*-algebra and let $a \in \cup_{n=1}^{\infty} M_n (A)_+$ be full. 
Let $r\in (0, \infty)$. Let $x, y \in \cup_{n=1}^{\infty} M_n (A)_+$
satisfy 
\begin{equation}
\label{EQ1.2025.03.10}
\widehat{[x]}  + r 
 < 
\widehat{[y]}
\quad
\mbox{ on } \cF_{[a]} (\Cu(A)),
\end{equation} then
$y$ is full, and, therefore, 
$\overline{\Span (A\otimes \cK) x (A\otimes \cK)} \subseteq \overline{\Span (A\otimes \cK) y (A\otimes \cK)}.
$
\end{lem}
\begin{proof}
By (\ref{EQ1.2025.03.10}), it is clear that $y\neq 0$. 
Now, set 
\[
\eta =\inf_{\lambda \in \cF_{[1_A]}(\Cu(A))} \lambda ([a] ) 
\quad
\mbox{and}
\quad
I =\overline{\Span (A\otimes \cK) y (A\otimes \cK)}.
\]
 Since $a$ is full, it follows from Lemma~\ref{ResiduallyAndFullEquivalent} that $\eta>0$. Now, we use this, (\ref{EQ1.2025.03.10}),  and 
Lemma~\ref{LemcFbigection} to get
\[
0<r \eta \leq  \widehat{[x]}  + r \cdot  \widehat{[a]}
 < 
\widehat{[y]}
\quad
\mbox{ on } \cF_{[1_A]} (\Cu(A)).
\]
This implies that 
$\inf_{\lambda \in \cF_{[1_A]}(\Cu(A))} \lambda  ([y])>0$. Then, by Lemma~\ref{ResiduallyAndFullEquivalent}, $y$ is  full.  
\end{proof}
For unital C*-algebra $A$ and a full positive element $a \in \cup_{n=1}^{\infty} M_n (A)_+$, one may ask about the relationship between 
$\Rrc(A, a)$ and $\rc(\Cu(A), [a])$. In particular, about the $r$-comparison of $(\Cu(A), [a])$ and $(\W(A), [a])$. We discuss this in the following lemma and the theorem that follows. 
\begin{lem}
\label{PropR1234}
Let $A$ be a unital residually stably finite C*-algebra and let $a \in \cup_{n=1}^{\infty} M_n (A)_+$ be full. 
Let $r\in (0, \infty)$. 
We set the following properties of $r$:
\begin{itemize}
\item[$R1$]
 If $x, y \in \cup_{n=1}^{\infty} M_n (A)_+$
satisfy 
$
\widehat{[x]}  + r \cdot \widehat{[a]}
 <
\widehat{ [y]}$
 on $\cF_{[1_A]} (\Cu(A))$,
then $x \precsim_A y$.
\item[$R2$]
 If $x, y \in \cup_{n=1}^{\infty} M_n (A)_+$
satisfy 
$
\widehat{[x]}  + r 
 <
\widehat{[y]}$
 on $\cF_{[a]} (\Cu(A))$,
then $x \precsim_A y$.
\item[$R3$]
 If $x, y \in (A \otimes \cK)_+$
satisfy 
$
\widehat{ [x]}  + r \cdot \widehat{[a]}
 <
\widehat{[y]}$
 on $\cF_{[1_A]} (\Cu(A))$,
then $x \precsim_A y$.
\item[$R4$]
 If $x, y \in (A \otimes \cK)_+$
satisfy 
$
\widehat{ [x]}  + r 
 \leq
\widehat{[y]}$
 on $\cF (\Cu(A))$,
then $x \precsim_A y$.
\end{itemize}
Then:
\begin{enumerate}
\item
\label{PropR1234.a}
$r$ satisfies $R1$, $R2$, and $R3$ if and only if any one of them holds.
\item
\label{PropR1234.b}
If $r$ satisfies $R4$, then $r$ satisfies $R3$ (and hence R2 and R3). 
\item
\label{PropR1234.c}
If $r$ satisfies $R3$, then $r+\ep$ satisfies R4 for every $\ep>0$.
\end{enumerate} 
\end{lem}
\begin{proof}
Let $r \in (0, \infty)$. 
We first prove the implication  ``$r$ satisfies $R1 \Rightarrow r$ satisfies $R2$". Let $x, y \in \cup_{n=1}^{\infty} M_n (A)_+$
satisfy 
$
\widehat{[x]}  + r 
 <
\widehat{[y]}$
 on
$\cF_{[a]} (\Cu(A))$. 
For any $\lambda \in \cF_{[1_A]} (\Cu(A))$, by Lemma~\ref{LemcFbigection}, we have $\tfrac{1}{\lambda([a])} \cdot \lambda \in \cF_{[a]}(\Cu(A))$. Using this, we get
\begin{equation}\label{Eq3.2025.03.10}
\tfrac{1}{\lambda([a])} \cdot \lambda([x]) + r < \tfrac{1}{\lambda([a])}\cdot  \lambda ([y])
\quad
\mbox{for all } \lambda \in \cF_{[1_A]} (\Cu(A)).
\end{equation}
Multiply both sides of (\ref{Eq3.2025.03.10}) by $\lambda([a])$ to get
$
 \widehat{ [x] } + r \cdot \widehat{ [a]} < \widehat{[y]}
 $
 on
$\cF_{[1_A]} (\Cu(A))$.
Since $r$ satisfies $R1$, it follows that $x \precsim_A y$, as desired. 

The proof of the implication ``$r$ satisfies $R2 \Rightarrow r$ satisfies $R1$" is similar to the previous part, using the same method and Lemma~\ref{LemcFbigection}. 

Second, we prove ``$r$ satisfies $R3 \Leftrightarrow r$ satisfies $R1$". The forward implication is trivial. So, we prove the implication 
``$r$ satisfies $R1 \Rightarrow r$ satisfies $R3$". Let
$x, y \in (A \otimes \cK)_+$
satisfy 
\begin{equation}
\label{Eq6.2025.03.10}
\widehat{[x]} + r \cdot \widehat{[a]}
 <
\widehat{[y]} 
\quad \mbox{  on }  \cF_{[1_A]} (\Cu(A)).
\end{equation}
We may assume that $\| x\| \leq 1 $.
 Let $\ep \in (0, \infty)$.  By Lemma~\ref{PhiB.Lem.18.4}(\ref{PhiB.Lem.18.4.11}), it suffices to show that there exists $\dt>0$ such that $(x - \ep)_+ \precsim_A (y-\dt)_+$. 
Define the continuous function $g \colon [0, \infty) \to [0, 1]$ give by 
\[
g(t) =
\begin{cases} 
\ep^{-1} t, & \text{if } 0\leq t < \ep \\
1, & \text{if } t \geq \ep.
\end{cases}
 \]
 Since $g(x) (x-\ep)_+ =(x-\ep)_+$, it follows 
 from Lemma~\ref{LemTrace} that,
 for every $\tau \in \QT_{2}^{1} (A)$, 
 \begin{equation}\label{Eq1.2025.04.09}
 d_{\tau} ((x- \ep)_+) \leq \tau (g(x)) \leq d_{\tau} (x). 
\end{equation}
 Using this, (\ref{Eq6.2025.03.10}), and Remark~\ref{rmkHandelElliott}, we get
 \begin{equation*}
 \tau ( g(x) ) + r \cdot d_{\tau} (a)
 <
d_{\tau} (y) \quad \mbox{ for all } \tau \in \QT_{2}^{1} (A).
\end{equation*} 
 We choose $k \in \N$ such that $a \in M_k (A)$. For all $\tau \in \QT_{2}^{1} (A)$, 
 we have 
\begin{equation}\label{Eq2.2025.04.09}
 d_{\tau} (a) \leq k.
 \end{equation} 
 Also the map $\tau \mapsto \tau (g(x)) + rk$ is a continuous map from $\QT_{2}^{1} (A)$ to $[0, \infty)$. 
 For every $n\in \N$, we know that
  $(y - 1/n)_+ \sim c$ where $c \in M_{n_0} (A)$ for some m in $n_0 \in \N$. So, $c \leq \| c\| \cdot 1_{M_{n_0} (A)}$. Therefore
 \[
 d_\tau ((y - 1/n)_+) = d_\tau (c) \leq d_\tau (1_{M_{n_0} (A)}) = n_0<\infty
 \quad
 \mbox{for all }\tau \in \QT_{2}^{1} (A).
 \]
 Now, for every $n\in \N$, define
  $f_n \colon \QT_{2}^{1} (A) \to [0, \infty)$ by
  $f_n (\tau) = d_{\tau} ((y - \tfrac{1}{n})_+)$. 
  We know that 
  $(f_n)_{n=1}^{\infty}$ is an increasing sequence of lower semicontinuous functions,  and, for each $\tau \in \QT_{2}^{1} (A)$,
  we have $\lim_{n \to \infty} d_{\tau} ((y - \tfrac{1}{n})_+) = d_{\tau} (y)$. 
  Now, we use 
  Lemma~6.13 of \cite{Ph14} to choose 
  $n \in \N$ such that 
  \begin{equation}\label{Eq3.2025.04.03}
  \tau (g(x)) + rk < d_{\tau} ((y - \tfrac{1}{n})_+)
\quad
\mbox{for all } \tau \in \QT_{2}^{1} (A).
 \end{equation}  
   Therefore, for all $\tau \in \QT_{2}^{1} (A)$,
  \[
   d_{\tau} ((x- \ep)_+) + r \cdot d_{\tau} (a)
   \underset{\text{(\ref{Eq1.2025.04.09}) and (\ref{Eq2.2025.04.09}) }}{\leq}
     \tau (g(x)) + rk 
 \underset{\text{(\ref{Eq3.2025.04.03})}}{<}    
     d_{\tau} ((y - \tfrac{1}{n})_+).
  \]
 So
 $\widehat{[(x- \ep)_+]} + r \cdot \widehat{[a]} <  \widehat{[(y - \tfrac{1}{n})_+]}$
   on $\cF_{[1_A]} (\Cu(A))$. 
 Using  this and the fact that $r$ satisfies $R3$, we have $(x- \ep)_+ \precsim (y - \tfrac{1}{n})_+$. Now, put $\dt= \tfrac{1}{n}$. This complete the proof of this part. 

Now, we prove (\ref{PropR1234.b}). So let $r>0$ satisfy $R4$ and let
$x, y \in (A \otimes \cK)_+$
satisfy 
\begin{equation}
\label{Eq12.2025.03.10}
\lambda ( [x] ) + r \cdot \lambda ([a])
 <
\lambda( [y]) \quad
\mbox{for all } \lambda \in \cF_{[1_A]} (\Cu(A)).
\end{equation}
This implies that
\begin{equation}
\label{Eq10.2025.03.10}
\lambda ( [x] ) + r \cdot \lambda ([a])
 <
\lambda( [y])
\quad
\mbox{
for all } \lambda \in \cF (\Cu(A)) 
\mbox{ with } 0<\lambda ([1_A])<\infty.
\end{equation}
For $\lambda \in \cF(\Cu(A))$ with 
$\lambda ([1_A]) =0$, by Lemma~\ref{LemcFbigection}(\ref{LemcFbigection.a}), we have $\lambda ([a])=0$, hence $\lambda \equiv 0$. 
Now, suppose that $\lambda \in \cF (\Cu(A))$
with $\lambda([1_A])=\infty$.
Since $y$ satisfies (\ref{Eq12.2025.03.10}) and  $A$ is residually stably finite, it follows from (\ref{PropR1234.a}) and Lemma~\ref{LemFullRrc} that $y$ is full and, therefore, $\lambda ([y])=\infty$. 
Putting these together, we get
$
\widehat{ [x] } + r \cdot \widehat{[a]}
 \leq
\widehat{[y]}$ 
 on   $\cF (\Cu(A))$. 
Since $r$ satisfies  $R4$, it follows that $x \precsim y$, as desired.

Finally, we prove (\ref{PropR1234.c}). 
So  let
$x, y \in (A \otimes \cK)_+$
satisfy 
\begin{equation}
\label{Eq7.2025.03.10}
\widehat{[x]} + (r+\ep) \cdot \widehat{[a]}
 \leq 
\widehat{[y]} \quad \mbox{  on }  \in \cF (\Cu(A)).
\end{equation}
Let $\ep>0$. It suffices to show that $(x-\ep)_+ \precsim y$. By (\ref{Eq7.2025.03.10}), we have 
\begin{equation}
\label{Eq8.2025.03.10}
\lambda ( [x] ) + (r+\ep) \cdot \lambda ([a])
 \leq
\lambda( [y]) \quad \mbox{ for all } \lambda \in \cF_{[1_A]} (\Cu(A)).
\end{equation}
Since $\lambda ( [(x-\ep)_+] ), \lambda ( [a] )< \infty$, it follows from (\ref{Eq8.2025.03.10}) that 
\[
\lambda ( [(x-\ep)_+ ) + r \cdot \lambda ([a])
 <
\lambda( [y]) \quad \mbox{ for all } \lambda \in \cF_{[1_A]} (\Cu(A)).
\]
Since $r$ satisfies $R3$, we get $(x-\ep)_+ \precsim y$, as desired. 

\end{proof}
In Subsection~3.1 of \cite{BRTW12},  where an algebraic equivalent for
 the conventional radius of comparison is given, the reason provided there
 seems not to be correct. This issue was addressed 
 for unital simple stably finite C*-algebra in Proposition~6.12 \cite{Ph14}. However, there seems to be a misprint in the proof of that proposition as well. 
 So we encourage the reader to take 
 $a = 1_A$ and follow
 the proof of
 Lemma~\ref{PropR1234}(\ref{PropR1234.a}), 
 as this provides a revised explanation for the issue. 
 
Putting Lemma~\ref{PropR1234}, Definition~\ref{RrcDf}, and Definition~\ref{DefRcCuAlgebraic} together, we get the following theorem as promised. 
\begin{thm}\label{CorRrc=rc}
Let $A$ be a unital residually stably finite C*-algebra and let $a \in \cup_{n=1}^{\infty} M_n (A)_+$ be full. Then
$\Rrc(A, a) = \rc (\Cu (A), [a])$. 
\end{thm}
In the above theorem if $a =1_A$, then it is Proposition~3.2.3 of \cite{BRTW12}. 

We conclude this section by the following remark. 
\begin{rmk}\label{RmkRCHeredit}
The relative radius of comparison provides  us the relative radius of comparison of full hereditary C*-subalgebras as well. Namely,
for a full positive element $a$ in a C*-algebra $A \otimes \cK$, we know that  the hereditary C*-subalgebra of $A\otimes \cK$ generated by $a$, i.e., 
$\overline{a (A \otimes \cK) a}$,   is  also full.  So 
 $\rc (\Cu \big(\overline{a (A \otimes \cK) a}\big), [a])= \rc (\Cu (A), [a])$. 
 
\end{rmk}

\section{The Rank Ratio Function}
In this section, 
we give the definition of the rank ratio function, associated to a Cuntz semigroup, and its normalized version of it. Then we study the most important  properties of them and their relationships.  

Here is the definition of the rank ratio function associated to a $\Cu$-semigroup. 
\begin{dfn}
\label{rk_rDef}
Let $S$ be a $\Cu$-semigroup and let $x, y \in S$.
Define
\[
\cR(\widehat{x}, \widehat{y})= \big\{ r \in (0, \infty) \colon \widehat{x} \leq r\cdot \widehat{y} \big\}.
\] 
\begin{enumerate}
\item
The \emph{rank ratio of $(x, y)$}, denoted by $\rho (x, y)$, is 
\[
\rho (x, y)= 
 \inf \cR(\widehat{x}, \widehat{y}).
 \]
 If there is no such $r$, then $\rho(x, y)=\infty$.
 \item 
The \emph{rank ratio function} is the function
$\rho \colon S \times S \to [0, \infty]$ given by $(x, y) \mapsto \rho(x, y)$.
\end{enumerate}
\end{dfn}
The rank ratio function at $(x, y)$ provides us the smallest $r$ in $(0, \infty)$ that ensures a global control of the rank of $x$ relative to rank of $y$. Such $r$ does not always exist (see Lemma~\ref{lemRhoInvese}). 

For a $\Cu$-semigroup $S$ and $z\in \Fu (S)$, we denote by
$S_{\ll \infty \cdot z}$ the set of all $x \in S$ with 
$x \ll \infty \cdot z$. In Definition~\ref{rk_rDef} if we only consider the normalized functionals on $S$, then
this gives us the normalized rank ratio function associated to a $\Cu$-semigroup. 
\begin{dfn}
\label{rk_r_0Def}
Let $S$ be a $\Cu$-semigroup, let $z \in \Full(S)$, and let $\cF_z(S) \neq \emptyset$. For any $x, y \in S$ with $x, y \ll \infty\cdot z$. 
Define
\[
\cR_{z} (\widehat{x}, \widehat{y})=
\left\{ r \in (0, \infty) \colon \widehat{x} \leq r\cdot \widehat{y} \quad \mbox{ on} \ \cF_{z}(S) \right\}.
\]
\begin{enumerate}
\item
The \emph{normalized rank ratio of $(x, y)$}, denoted by $\rho_z (x, y)$, is 
\[
\rho_z (x, y)= 
 \inf \cR_{z} (\widehat{x}, \widehat{y}).
\]
 If there is no such $r$, then $\rho_z(x, y)=\infty$. 
 \item
The \emph{normalized rank ratio function} is the function
$\rho_z \colon S_{\ll \infty \cdot z} \times S_{\ll \infty \cdot z} \to [0, \infty]$ given by $(x, y) \mapsto \rho_z(x, y)$.
\end{enumerate}
\end{dfn}

The following lemma provides a general relationship between 
$\rho$ and its normalized version. 
\begin{lem}\label{rhorho0}
Let $S$ be a $\Cu$-semigroup, let $z \in \Full(S)$, and let $\cF_z(S) \neq \emptyset$. Let $x, y \in S$ with $x, y \ll \infty\cdot z$. 
Then
$\rho_z (x, y) \leq \rho(x, y)$.
\end{lem}
\begin{proof}
If $\rho(x, y)= \infty$, then the inequality is trivially true. 
So, assume that  $\rho(x, y)<\infty$. 
For every $\ep>0$, we choose $r>0$ in $\cR(\widehat{x}, \widehat{y})$ such that 
$r < \rho(x, y)+\ep$. So we have 
$\widehat{x} \leq r\cdot \widehat{y}$   on $\cF(S)$, therefore, on $\cF_z (S)$. 
This implies that  $\rho_z (x, y) \leq r < \rho(x, y)+ \ep$. Therefore, for every 
$\ep >0$, we have 
$\rho_z (x, y) \leq \rho(x, y)+ \ep$.  
\end{proof}

The following lemma provides some basic facts about the rank ratio function. 
\begin{lem}\label{RkRatioProperty}
Let $S$ be a $\Cu$-semigroup and  let $x, y \in S$ with $x\leq y$. 
Then:
\begin{enumerate}
\item\label{RkRatioProperty.a_0}
If $x=0$ and $y \neq 0$,  then 
$\rho (x, y)= 0$ and $\rho(y, x)=\infty$. 
 \item\label{RkRatioProperty.b}
 $\rho(x, y) \leq 1$.
If further, $\cF_y(S) \neq \emptyset$ and $x \in \Full (S)$, then 
 $0<\rho(x, y) \leq 1$. 
 \item\label{RkRatioProperty.c}
 If $\cF_x(S) \neq \emptyset$,   then 
 $\rho(y, x)\geq 1$. 
 \item\label{RkRatioProperty.e}
 If 
$\rho(y, x)=1$, then $\widehat{x} = \widehat{y}$. 
\item\label{RkRatioProperty.f}
Assume that $x \in \Full(S)$, further $\cF_{x} (S) \neq \emptyset$,  and  $y \ll \infty \cdot x$, then 
$1\leq \rho(y, x)\leq k$ for some $k \in \N$. 
\end{enumerate}
\end{lem}
\begin{proof}
We prove (\ref{RkRatioProperty.a_0}).
Since $0\leq r \cdot \widehat{y}$  for every $r \in (0, \infty)$, it follows from the definition of $\rho$ that 
$\rho (x, y)=0$. To prove the second part, assume that 
$\rho (y, x) <\infty$. So, there exits $r>0$ in $\cR(\widehat{y}, \widehat{x})$.
Therefore,
$\lambda (y) \leq r \cdot \lambda (x)$ for all $\lambda \in \cF(S)$. Using this and taking $\lambda$ to be $\lambda_\infty$, we get 
$\infty=\lambda_\infty (y) \leq r \cdot \lambda_\infty (x)=0$.
This is a contradiction.  

Now, we prove (\ref{RkRatioProperty.b}).
Since $x \leq y$, we have $\lambda (x) \leq \lambda (y)$ for all $\lambda \in \cF(S)$. So
$\rho(x, y)\leq 1$. 
To prove the second part, 
assume $\rho(x, y)=0$. So, 
it follows from Lemma~\ref{Rkiszero} that $\cF_{y} (S)=\emptyset$. This is contradiction. 

We prove (\ref{RkRatioProperty.c}). 
If $\rho(y, x)=\infty$, the result is trivially true. so assume that 
$\rho (y, x) < \infty$. 
For every $r \in \cR(\widehat{y}, \widehat{x})$, we have 
\begin{equation}\label{Eq11.2025.04.09}
\lambda (y) \leq r \cdot \lambda (x)
\quad
\mbox{ for all } \lambda \in \cF(S).
\end{equation}
Since 
$\cF_x (S)\neq \emptyset$, choose $\lambda_0 \in \cF_x (S)$. Therefore, using $x \leq y$ at the second step,
\[
1= \lambda_0 (x) \leq \lambda_0 (y) 
\underset{\text{(\ref{Eq11.2025.04.09})}}{\leq}    
 r \cdot
\lambda_0 (x) = r. 
\]
This implies that $r\geq 1$ for all $r \in \cR(\widehat{y}, \widehat{x})$, and therefore $\rho(y, x) \geq  1$. 

We prove (\ref{RkRatioProperty.e}). 
Since $\rho(y, x)=1$,
for every $\ep>0$, we choose $r_{\ep}>0$ in $\cR(\widehat{y}, \widehat{x})$ such that $r_{\ep}< 1+\ep$. 
So, for every $\ep>0$, we have 
\begin{equation}\label{Eq22.2025.04.09}
\lambda (y) \leq (1+\ep)\cdot \lambda (x)
\mbox{ for all }  \lambda\in \cF(S).
\end{equation}
\begin{itemize}
\item
If $\lambda(x)=0$, then, by (\ref{Eq22.2025.04.09}), we have $\lambda(y)=0$.
\item
If $\lambda(x)=\infty$, then, by $x\leq y$, we have $\lambda (y)=\infty$.
\item
If $0<\lambda(x)<\infty$, then, by $x\leq y$, we have, for every $\ep>0$, 
\[
\lambda (x)\leq\lambda (y) \leq (1+\ep)\cdot \lambda (x).
\]
This implies that $\lambda (x)= \lambda (y)$. 
\end{itemize}
Therefore, $\lambda (x)= \lambda (y)$
for all $\lambda\in \cF(S)$. 

We prove (\ref{RkRatioProperty.f}). Since 
$y \ll \infty \cdot x = \sup_n n\cdot x$, we choose $k \in \N$ such that $y \leq k x$. So
$\lambda (y) \leq k \lambda (x)$
 for all  $\lambda \in \cF(S)$. Now, using this and 
 (\ref{RkRatioProperty.c}), we get 
 $1\leq \rho(y, x)\leq k$ for some $k \in \N$. 
\end{proof}
One may ask when $\rho(x, y)$ can be zero for $x, y \in \Fu(S)$. 
The following lemma answers this question. 
\begin{lem}
\label{Rkiszero}
Let $S$ be a $\Cu$-semigroup and let $x, y \in \Fu (S)$. Then the following are equivalent:
\begin{enumerate}
\item \label{Rkiszero.a}
$\rho(x, y)= 0$.  
\item\label{Rkiszero.c}
$\cF_{y} (S)= \emptyset$.
\end{enumerate}
\end{lem}
\begin{proof}
We claim that 
$\rho(x, y)= 0$ if and only if,
for every $n \in \N$, we have
$\lambda (x) \leq \tfrac{1}{n} \cdot \lambda (y)$ for all $\lambda \in \cF(S)$.
We prove the forward implication.
Since $\rho(x, y)=0$, it follows from the definition of $\rho$ that for every $n \in \N$ there exists $r_n \in \cR(\widehat{x}, \widehat{y})$ such that $r_n < \frac{1}{n}$. So, we have 
$\lambda (x) \leq r_n \cdot \lambda (y) \leq \frac{1}{n} \cdot \lambda (y)$ for all $\lambda \in \cF(S)$. 
The backward implication, 
 follows immediately from the definition of $\rho(x, y)$. 

Now, we prove the implication ``(\ref{Rkiszero.a}) $\Rightarrow$ (\ref{Rkiszero.c})". Suppose that $\cF_{y} (S)\neq  \emptyset$. We choose $\lambda_0 \in \cF_{y}(S)$ to get
$\lambda_0 (x) \leq \frac{1}{n}$ for all $n \in \N$. So, $\lambda_0 (x)=0$ and, since $x$ is full, we have $\lambda_0 \equiv 0$. This contradicts 
$\lambda_0 \in \cF_{y}(S)$. 

We prove implication  ``(\ref{Rkiszero.c})  
$\Rightarrow$ (\ref{Rkiszero.a})".
Since $\cF_{y}(S) = \emptyset$, we have 
$\lambda (y) \in \{0, \infty\}$ for all $\lambda \in \cF(S)$. 
\begin{enumerate}
\item[Case 1]
 If $\lambda(y)=0$, then, by $y \in \Fu(S)$, we have $\lambda(x)=0$. 
\item[Case 2]
 If $\lambda(y)=\infty$, then, for every $n \in \N$, we have $\frac{1}{n} \cdot \lambda (y)= \infty$. So, $\lambda(x) \leq \frac{1}{n} \cdot \lambda (y)$.
\end{enumerate}
Putting Case 1 and Case 2 together,
we get 
$\lambda(x) \leq \frac{1}{n} \cdot \lambda (y)$
for all $\lambda \in \cF(S)$, 
and, therefore, (\ref{Rkiszero.a}) follows. 
\end{proof}

For a purely infinite C*-algebra $A$,
 Lemma~\ref{Rkiszero}
shows that $\rho ([a], [b])=0$ 
for all $[a], [b] \in \Fu (\Cu (A))$
(see also Example~\ref{Exa.PurelyInfinite}). 
\begin{lem}\label{lemRhoInvese}
Let $S$ be a $\Cu$-semigroup and let $x, y, \in \Fu (S)$.
Suppose that $\rho(x, y)=0$. Then
 \[
\rho( y, x)=
\begin{cases}
        0, & \text{if }  \cF_x (S) = \emptyset\\
        \infty, & \text{if }\cF_x (S) \neq \emptyset.
    \end{cases}
\]
\end{lem}
\begin{proof}
Since $\rho(x, y)=0$, 
it follows from Lemma~\ref{Rkiszero} that 
$\cF_{y} (S)= \emptyset$.  
Now, if $\cF_x (S) = \emptyset$, then it follows from Lemma~\ref{Rkiszero} that $\rho (y, x)= 0$.
So let $\cF_x (S) \neq \emptyset$. We show that $\rho(y, x)= \infty$. Assume that $\rho(y, x)<\infty$. 
So, there exists $r$ in 
$\cR(\widehat{y}, \widehat{x})$
such that 
$\lambda (y) \leq r \cdot \lambda(x)$ for all $\cF(S)$.
Choose $\lambda_0 \in \cF_x(S)$. So $\lambda_0 (y)\leq r$.
This is a contradiction. Because 
 we know that $\lambda_0 (y)$ cannot be zero as $y \in \Fu (S)$. Also, $\lambda_0 (y)$ cannot be $\infty$. 
\end{proof}
There is no reason to expect that $\rho (x, y) = \frac{1}{\rho(y, x)}$ for full elements $x, y$ in a $\Cu$-semigroup $S$ (see Example~\ref{Exa.RecipricalValue} or Proposition~\ref{rceveryGroup}(\ref{rceveryGroup.a})). 

The following lemma is about monotonicity and homogeneity of the rank ratio function. 
\begin{lem}\label{IncDesrho}
Let $S$ be a $\Cu$-semigroup and let $x, y, z\in S$. Then:
\begin{enumerate}
\item \label{IncDesrho.a}
$\rho(x, n \cdot y)= \frac{1}{n} \cdot \rho(x, y)$ 
and 
$\rho(n \cdot x, y)= n \cdot \rho(x, y)$ for $n\in \N$.
\item\label{IncDesrho.b}
If $x \leq y \leq z$, then
$\rho(z,  y) \leq \rho(z, x)$
and 
$\rho(x, z) \leq \rho(y, z)$.
\item\label{IncDesrho.c}
$\rho (x+y, z) \geq \max \left( \rho (x, z), \rho (y, z)\right)$
and 
$\rho (z, x+y) \leq \min \left( \rho (z, x), \rho (z, y)\right)$. 
\end{enumerate}
\end{lem}
\begin{proof}
We  prove (\ref{IncDesrho.a}). 
It is easy to check that ``$\rho(x, y)=\infty
\Leftrightarrow \rho(x, n \cdot y)=\infty \Leftrightarrow
\rho(n\cdot x, y)=\infty$" for $n\in \N$. 
So, we suppose that $\rho(x, y) < \infty$. 
For every $n \in \N$, we have $\lambda(n \cdot x)=n \cdot \lambda(x)$ and  $\lambda(n\cdot y)=n \cdot \lambda(y)$
for all $\lambda \in \cF(S)$. So, for $r \in (0, \infty)$,  we have
\begin{align*}
r \in  \cR(\widehat{x}, \widehat{n\cdot y}) &\Leftrightarrow 
\lambda (x) \leq r \cdot \lambda (n\cdot y) \ \ \forall
\lambda \in \cF(S)
\\&
\Leftrightarrow 
\lambda (x) \leq nr \cdot \lambda (y) \ \ \forall
\lambda \in \cF(S)
\Leftrightarrow 
nr \in  \cR(\widehat{x}, \widehat{y}).
\end{align*}
and 
\begin{align*}
r \in  \cR(\widehat{n\cdot x}, \widehat{y}) &\Leftrightarrow 
\lambda (n\cdot x) \leq r \cdot \lambda (y) \ \ \forall
\lambda \in \cF(S)
\\&
\Leftrightarrow 
n\cdot\lambda (x) \leq r \cdot \lambda (y) \ \ \forall
\lambda \in \cF(S)
\Leftrightarrow 
\frac{r}{n} \in  \cR(\widehat{x}, \widehat{y}).
\end{align*}
Therefore, 
 $r \in  \cR(\widehat{x}, \widehat{n\cdot y})$ if and only if $r \in \frac{1}{n} \cdot  \cR(\widehat{x}, \widehat{y})$ and, also, $r\in  \cR(\widehat{n\cdot x}, \widehat{y})$ if and only if $r \in n \cdot  \cR(\widehat{x}, \widehat{y})$. Using the definition of $\rho(x, n\cdot y)$ and $\rho(n\cdot x, y)$, the result follows. 
 
 We prove (\ref{IncDesrho.b}).
 We only prove $\rho (z, y) \leq \rho (z, x)$. The proof of the second  part is very similar to the first part. 
First, we claim that if $\rho(y, z)=\infty$, then $\rho (z, x)=\infty$. 
Assume that $\rho (z, x)<\infty$.
So, there exists $r>0$ in $ \cR(\widehat{z}, \widehat{x})$ such that 
$\widehat{z} \leq r\cdot \widehat{x}$.  Since $x \leq y$, we have 
$\widehat{z} \leq r\cdot \widehat{y}$.  This implies that $\rho (z, y)< \infty$. This is a contradiction. 
Now suppose that $\rho(z, x) < \infty$. 
Let $r> 0$ be in $ \cR(\widehat{z}, \widehat{x})$. So $\widehat{z} \leq r\cdot \widehat{x}$.
Using $x\leq y$, we get $\widehat{z} \leq r\cdot \widehat{y}$.
Thus $r\in \cR(\widehat{z}, \widehat{y})$. 
This shows that $ \cR(\widehat{z}, \widehat{x}) \subseteq \cR(\widehat{z}, \widehat{y})$ and therefore
\[
\rho(z, y)=\inf \cR(\widehat{z}, \widehat{y}) \leq \inf  \cR(\widehat{z}, \widehat{x})=\rho(z, x).
\]

Part (\ref{IncDesrho.c}) follows from (\ref{IncDesrho.b}) and the fact that $x \leq x+y$ and $y \leq x+y$. 
This completes the proof. 
\end{proof}

Lemma~\ref{IncDesrho}(\ref{IncDesrho.c})
shows that $\rho$ is  (averaged) subadditive 
in the second argument and is 
(averaged) superadditive 
in the first argument.

The following lemma determines  the rank ratio structure of an arbitrary increasing sequence in a Cuntz semigroup. 
\begin{thm}\label{LimofRank}
Let $S$ be a $\Cu$-semigroup and 
let $x$ and a sequence $(x_n)_{n=1}^{\infty}$ be in $S$ with $x_n\leq x_{n+1} \leq x$ for all $n \in \N$. 
Then:
\begin{enumerate}
\item\label{LimofRank.a}
The sequence  
$(\rho(x_n, x))_{n=1}^{\infty}$ is an increasing sequence in $\mathbb{R}$.,  bounded  above by 1. So, $\lim_{n \to \infty} \rho(x_n, x)$ exists in $\mathbb{R}$.

If further $x_n \in \Full (S)$ for all $n \in \N$ and $\cF_x(S)\neq \emptyset$, then 
$0<\lim_{n \to \infty} \rho(x_n, x)$.
\item \label{LimofRank.b}
The sequence  
$(\rho(x, x_n))_{n=1}^{\infty}$ in an decreasing  sequence in $[0, \infty]$. 

If further $x_n \in \Full (S)$ for all $n \in \N$ and $\cF_x(S)\neq \emptyset$, then 
$\lim_{n \to \infty} \rho(x, x_n) \geq 1$. 
\end{enumerate}
\end{thm}
\begin{proof}
We prove (\ref{LimofRank.a}).
Since $x_n \leq x_{n+1} \leq x$, it follows from Lemma~\ref{IncDesrho}(\ref{IncDesrho.b}) and  Lemma~\ref{RkRatioProperty}(\ref{RkRatioProperty.b}) that 
$
\rho (x_n, x) \leq \rho (x_{n+1}, x) \leq 1
$
for all  $n \in \N$.
So the first part of (\ref{LimofRank.a}) follows. Now,
if further $\cF_{x} (S) \neq \emptyset$, then, by Remark~\ref{Fulfunctional} and the fact that $x_n\in \Fu (S)$, we have 
$\cF_{x_n} (S) \neq \emptyset$, 
and, therefore,   $\rho(x, x_n)> 0$
for all $n \in \N$. 
 Using this and the fact that 
$(\rho(x_n, x))_{n=1}^{\infty}$ is an increasing sequence, we get $0<\lim_{n \to \infty} \rho(x_n, x)$. 

We prove (\ref{LimofRank.b}).
Since $x_n \leq x_{n+1} \leq x$, it follows from Lemma~\ref{IncDesrho}(\ref{IncDesrho.b})  that 
$
\rho (x, x_n) \geq \rho (x, x_{n+1})
$
 for all $n \in \N$.
So, the fist part of (\ref{LimofRank.b}) follows. 
If further 
$\cF_{x} (S) \neq \emptyset$, then, by Lemma~\ref{RkRatioProperty}(\ref{RkRatioProperty.c}),
$
\rho (x, x_n) \geq \rho (x, x_{n+1})\geq 1 
$
for all $n \in \N$,
and, therefore, the second part of (\ref{LimofRank.b}) follows. 
\end{proof}
In Part (\ref{LimofRank.b}) of the above theorem
the condition $\cF_x(S) \neq \emptyset$ is necessary. (See Example~\ref{ExprhoofNA}(\ref{ExprhoofNA.b}).)
See Example~\ref{ExprhoofNA} for a more concrete example satisfying Theorem~\ref{LimofRank}.

In the following proposition, we give an equivalent version of the definition of the normalized rank ratio function and, more importantly, we show that the rank ratio function and normalized rank ratio function are equal on the full elements of unital residually stably finite C*-algebras. 
\begin{prp}\label{Prp.rhoequiavalent}
Let $A$ be a unital stably finite C*-algebra and let $a, b \in \cup_{n=1}^{\infty} M_n (A)_+$. 
Set $\eta= \inf_{{\lambda \in \cF_{[1_A]} (\Cu(A)) }} \lambda (b)$. Then
\begin{enumerate}
\item\label{Prp.rhoequiavalent.b}
If $a$ is full and $A$ is residually stably finite, then $\rho ([a], [b]) = \rho_{[1_A]} ([a], [b])$.
\item\label{Prp.rhoequiavalent.a_0}
If $b$ is full, then 
$\eta>0$, and, further
\[
\sup_{\lambda \in \cF_{[1_A]} (\Cu(A)) } \left( \frac{\lambda ([a])}{\lambda ([b])} \right) \leq \frac{k}{\eta} \
 \mbox{ and } \
\rho_{[1_A]} ([a], [b]) \leq \frac{k}{\eta} \
\mbox{ for some } k \in \N.
\] 
\item\label{Prp.rhoequiavalent.a}
If $b$ is full , then
$
\rho_{[1_A]} ([a], [b])  =
\sup_{\lambda \in \cF_{[1_A]} (\Cu(A)) } \left( \frac{\lambda ([a])}{\lambda ([b])} \right).
$
\item\label{Prp.rhoequiavalent.c}
If $a, b$ are full and 
and $A$ is residually stably finite, then 
\[
\rho ([a], [b]) = \rho_{[1_A]} ([a], [b]) = \sup_{\lambda \in \cF_{[1_A]} (\Cu(A)) } \left( \frac{\lambda ([a])}{\lambda ([b])} \right)
\leq \frac{k}{\eta} 
\]  
for some $k \in \N$.
\end{enumerate}
\end{prp}
\begin{proof}
We set $\Delta = \sup_{\lambda \in \cF_{[1_A]} (\Cu(A)) }  \left( \frac{\lambda ([a])}{\lambda ([b])} \right)$.
We prove (\ref{Prp.rhoequiavalent.b}).
By Lemma~\ref{rhorho0}, we have $\rho_{[1_A]} ([a], [b]) \leq \rho ([a], [b])$.  
It is enough to prove that $\rho ([a], [b]) \leq \rho_{[1_A]} ([a], [b])$. We may assume that $\rho_{[1_A]} ([a], [b])< \infty$. 
Let $r>0$ be in $\cR_{[1_A]} (\widehat{[a]}, \widehat{[b]})$. So
\begin{equation}
\label{Eq1.2025.03.17}
\lambda ([a]) \leq r \cdot \lambda ([b])
\quad
\mbox{ 
for all } \lambda \in \cF_{[1_A]}(\Cu(A)).
\end{equation} 
\begin{itemize}
\item[Case 1]
Using (\ref{Eq1.2025.03.17}) and Lemma~\ref{LemcFbigection}(\ref{LemcFbigection.b}), we get 
$\lambda ([a]) \leq r \cdot \lambda ([b])$
for all $\lambda \in \cF(\Cu(A))$
with $0<\lambda ([1_A]) < \infty$.
\item[Case 2]
Suppose $\lambda \in \cF(S)$ with $\lambda ([1_A])=0$. So
$\lambda ([a])=\lambda ([b])=0$.
So, $\lambda([a]) \leq r \cdot \lambda([b])$.  
\item[Case 3]
Suppose $\lambda \in \cF(S)$ with $\lambda ([1_A])=\infty$, so $\lambda([a])=\infty$
by Lemma~\ref{LemcFbigection}(\ref{LemcFbigection.a}). 
 Now, we claim that $\lambda ([b]) = \infty$. To prove this it suffices to show that $[b]$ is full. 
Since $a$ is full, then, by Lemma~\ref{ResiduallyAndFullEquivalent}, 
$\inf_{\lambda \in \cF_{[1_A]}(\Cu(A)) } \lambda ([a])>0$. 
Using (\ref{Eq1.2025.03.17}), we get 
$\inf_{\lambda \in \cF_{[1_A]} (\Cu(A))} \lambda ([b])>0$.  So, by Lemma~\ref{ResiduallyAndFullEquivalent}, 
$b$ is also full and 
$\lambda ([b]) =  \infty$. 
Thus  $\lambda([a]) \leq r \cdot \lambda([b])$.  
\end{itemize}
Putting all the cases together, we get 
$\lambda ([a]) \leq r \cdot \lambda [b]
\mbox{ for all } \lambda \in \cF(S)$. 
This implies that $\rho([a], [b]) \leq r$, and therefore, $\rho([a], [b]) \leq \rho_{[1_A]} ([a], [b])$.

We prove (\ref{Prp.rhoequiavalent.a_0}). 
Since $b$ is full, it follows from Lemma~\ref{ResiduallyAndFullEquivalent} that $\eta>0$. 
Also, since 
$a \in \cup_{n=1}^{\infty} M_n (A)_+$ and $A$ is unital, we choose $k \in \N$ such that 
$a \leq \| a \| \cdot 1_{M_{k} (A)}$. 
So, for all $\lambda \in \cF_{[1_A]} (\Cu(A))$, we have $\lambda ([a]) \leq k$. 

Using this,  for all $\lambda \in \cF_{[1_A]} (\Cu(A))$, we get 
$\frac{\lambda ([a])}{\lambda ([b])} \leq \frac{k}{\eta}$. Therefore, $\Delta \leq \frac{k}{\eta}$ and, also,  $\rho_{[1_A]} ([a], [b]) \leq \frac{k}{\eta}$. 

We prove (\ref{Prp.rhoequiavalent.a}).
We first show that $\Delta \leq \rho_{[1_A]} ([a], [b])$. 
So let $r>0$ be in $\cR_{[1_A]} (\widehat{[a]}, \widehat{[b]})$. 
This implies that 
\begin{equation}\label{EQ3.2025.04.11}
\lambda ([a]) \leq r \cdot \lambda ([b])
\quad
\mbox{ for all } \lambda \in \cF_{[1_A]} (\Cu(A)).
\end{equation} 
Using the fact that  $\eta>0$ and dividing (\ref{EQ3.2025.04.11}) by $\tfrac{1}{ \lambda([b]) }$, we have 
$\frac{ \lambda ([a]) }{\lambda ([b])} \leq r$ for all $\lambda \in \cF_{[1_A]} (\Cu(A))$.
Thus, 
$\Delta \leq r$. Taking infimum over such $r$, we get 
$\Delta \leq \rho_{[1_A]} ([a], [b])$. 
To prove the converse, we know that 
$\frac{\lambda ([a])}{\lambda ([b])}  \leq \Delta < \infty$
for all $\lambda \in \cF_{[1_A]} (\Cu(A))$.  
This implies that,  for all $\lambda \in \cF_{[1_A]} (\Cu(A))$, we have
$\lambda ([a]) \leq \Delta \cdot \lambda ([b])$.
Thus, $\rho_{[1_A]} ([a], [b]) \leq \Delta$, as desired. 

Part (\ref{Prp.rhoequiavalent.c}) follows immediately from 
(\ref{Prp.rhoequiavalent.b}) and (\ref{Prp.rhoequiavalent.a}). 
\end{proof}
 We refer to Definition~3.1 in \cite{BP18}, 
 which is originally from \cite{GH76}, for the value $\beta (x, y)$ and its properties
 for specific type of elements in $S$. 
 It is easy to show that if  $x, y$ are two given elements in a Cuntz semigroup
 with $x \leq n \cdot y$ for some $n \in \N$, then $\rho(x, y) \leq \beta(x, y)$. 

The focus of the rest of this section is  on the continuity of the rank of a full element in  the Cuntz semigroup of a unitl C*-algebra,  its connection to the normalized rank ratio function, and 
oscillation a positive full element. 

The following lemma is used in the proof of Lemma~\ref{FunCal}. 
\begin{lem}\label{LemTrace}
Let $A$ be a unital C*-algebra and let $\tau \in \QT_{2}^{1}(A)$. Let
$0 \leq a, b \leq 1$. Assume that
$a b=b$. Then
$\tau(b) \leq 
d_{\tau} (b) \leq \tau (a) 
 \leq d_{\tau} (a)$.
\end{lem}
\begin{proof}
Since $\|a \|, \|b \| \leq 1$, it follows that 
$a \leq a^{1/n}$ and $b \leq b^{1/n}$
for all $n \in \N$. This implies that
\[
\tau (a) \leq \lim_{n \to \infty} \tau(a^{1/n}) = d_{\tau} (a)
\quad
\mbox{and}
\quad
\tau (b) \leq \lim_{n \to \infty} \tau(b^{1/n}) = d_{\tau} (b).
\]
So, it remains to show that $d_{\tau} (b) \leq \tau(a)$. Using $ab=b$, we get
$a b^{1/n} =b^{1/n}$ for all $n \in \N$. So, for all $n \in \N$, we have 
\begin{align*}
\tau (b^{1/n}) = \tau (a b^{1/n}) =
\tau (a^{1/2} b^{1/n} a^{1/2})
\leq 
\|b^{1/n} \| \cdot \tau (a)\leq \tau (a).
\end{align*}
This implies that $d_{\tau} (b) = \lim_{n \to \infty} \tau (b^{1/n}) \leq \tau (a)$.  
This completes the proof.
\end{proof}
\begin{ntn}
For every $\ep>0$, we define
$
f_{\ep}(t)\
=
\begin{cases}
0, & \text{if }  t \in [0, \tfrac{\ep}{2}] \\
\tfrac{2}{\ep} x -1,  & \text{if }  t \in [\tfrac{\ep}{2}, \ep] \\
1, & \text{if } t \in [\ep, \infty).
\end{cases}
$
\end{ntn}

The following lemma is used  in the proof of Lemma~\ref{Oscillation-Equvallence}.
\begin{lem}\label{FunCal}
Let $A$ be a unital C*-algebra and let
$a \in A_+$ with $\|a\| \leq 1$. Let $\tau \in \QT_{2}^{1}(A)$ and let $\ep>0$. Then:
\begin{enumerate}
\item\label{FunCal.1}
$f_{2\ep}(a) \sim (a-\ep)_+$. 
\item\label{FunCal.2}
$f_{\ep}(a) (a-\ep)_+ =(a-\ep)_+$.
\item \label{FunCal.3}
$f_{\delta} (a) f_{\ep} (a) = f_{\ep} (a)$ for all $\delta \in (0, \tfrac{\ep}{4}]$.
\item\label{FunCal.4}
$\tau(f_{\ep} (a))  \leq d_{\tau} ((a-\delta)_+) \leq \tau (f_{\delta} (a))$ 
for all $\delta \in (0, \tfrac{\ep}{8}]$.
\end{enumerate}
\end{lem}
\begin{proof}
Part (\ref{FunCal.1}), (\ref{FunCal.2}), (\ref{FunCal.3}) follow from functional calculus at $a$ and definition of $f$. 

To prove (\ref{FunCal.4}), let $\dt \in (0, \tfrac{\ep}{8}]$.  It follows from (\ref{FunCal.2}) that 
$f_{\delta} (a) (a- \delta)_+=(a- \delta)_+$.
Then, by Lemma~\ref{LemTrace}, 
$d_{\tau} ((a-\delta)_+) \leq \tau (f_\dt (a))$. 
Since 
$2\delta \leq \tfrac{\ep}{4}$, it follows
from (\ref{FunCal.3}) that 
$f_{2\dt} (a) f_\ep (a) = f_\ep (a)$.
Using Lemma~\ref{LemTrace}, we get 
$\tau (f_\ep (a)) \leq d_\tau (f_{2\dt} (a))$.
Putting these and the fact that $f_{2 \delta} (a) \sim (a-\delta)_+$ together, we get
\[
\tau (f_\ep (a)) \leq d_{\tau} (f_{2 \delta} (a)) = d_{\tau} ((a-\delta)_+) \leq 
\tau (f_{\delta} (a)).
\]
This completes the proof. 
\end{proof}

The following definition is the definition of the oscillation of a positive element in a unital C*-algebra which is first introduced in A.1 of \cite{ElGONi20}. We refer to \cite{FH24, Li24}  for more detailed results on oscillation zero and its connection with $\cZ$-stability and stable rank one. 
\begin{dfn}\label{OscilDfn}
Let $A$ be a unital C*-algebra with $\QT_{2}^1(A)\neq \emptyset$ and let 
$a \in \cup_{n=1}^{\infty} M_n(A)_+$. 
The  oscillation of $a$ on $\QT_{2}^{1} (A)$, denoted by $\omega (a)$, is 
\[
\omega (a) =
\lim_{n \to \infty} \sup_{\tau \in \QT^1_2 (A) } 
\left\{ 
d_{\tau} (a) -  \tau \big( f_{\tfrac{1}{n}} (a) \big) 
 \right\}.
\]
\end{dfn}

The following lemma is used in the proof of Theorem~\ref{ContRank}. 
\begin{lem}\label{Oscillation-Equvallence}
Let $A$ be a unital C*-algebra $\QT_{2}^1(A)\neq \emptyset$, let $a\in \cup_{n=1}^{\infty} M_n(A)_+$ with $\| a\| \leq 1$. 
 Then the following are equivalent:
\begin{enumerate}
\item\label{Oscillation-Equvallence.1}
$\omega(a)=0$.
\item\label{Oscillation-Equvallence.2}
$\lim_{n \to \infty} \sup_{\tau \in \QT^1_2 (A)} \left\{ d_{\tau} (a) -  d_{\tau} ((a-\tfrac{1}{n})_+) \right\}=0$.
\end{enumerate}
\end{lem}
\begin{proof}
We first prove the implication ``(\ref{Oscillation-Equvallence.2}) $\Rightarrow$ (\ref{Oscillation-Equvallence.1})". Suppose 
\begin{equation}\label{Eq1.2025.04.16}
\lim_{n \to \infty} \sup_{\tau \in \QT^1_2 (A)}  \left\{ d_{\tau} (a) -  d_{\tau} ((a-\tfrac{1}{n})_+) \right\}=0.
\end{equation}
 Using Lemma~\ref{FunCal}(\ref{FunCal.4}) with $1/n$ in place of $\delta$, we get
 \[
 d_{\tau} (a) -  \tau \left(f_{\tfrac{1}{n}}(a) \right)
 \leq
 d_{\tau} (a) -  d_{\tau} \left( (a- \tfrac{1}{n})_+ \right)
 \quad
 \mbox{ for all } \tau \in \QT_2^{1}(A).
 \]
 This implies that
 \[
 \sup_{\tau \in \QT^1_2 (A) }  \left\{d_{\tau} (a) -  \tau \left(f_{\tfrac{1}{n}}(a) \right) \right\}
 \leq
\sup_{\tau \in \QT^1_2 (A) }  \left\{ d_{\tau} (a) -  d_{\tau} \left(\big(a-\tfrac{1}{n}\big)_+\right)
\right\}.
 \]
 Therefore, using this and using (\ref{Eq1.2025.04.16}), 
 $\lim_{n \to \infty}\sup_{\tau \in \QT^1_2 (A) }  \left\{d_{\tau} (a) -  \tau (f_{\tfrac{1}{n}}(a) \right\}=0$ and, by Definition~\ref{OscilDfn}, we have $\omega(a)=0$. 
 
 We prove the implication ``(\ref{Oscillation-Equvallence.1}) $\Rightarrow$ (\ref{Oscillation-Equvallence.2})".
 Now assume that 
 \begin{equation}
 \label{Eq1.2025.05.08}
 \lim_{n \to \infty} \sup_{\tau \in \QT^1_2 (A) } \left\{ d_{\tau} (a) -  \tau ((a-\tfrac{1}{n})_+) \right\}=0.
 \end{equation}
  Set $g_n (\tau) = \tau \left( f_{\tfrac{1}{n}}(a) \right)$. So, by (\ref{Eq1.2025.05.08}),
 $g_n$ converges to $\widehat{a}$ uniformly on $\QT^1_2 (A)$. Set $n_k = 8/k$. So, we have
$
 g_{n_k} (\tau) = \tau \left( f_{\tfrac{1}{n_k}} (a) \right)$.
 Since $g_{n_k}$ is a subsequence of $g_n$, it follows that 
  $g_{n_k}$  also converges to $\widehat{a} (\tau)$ uniformly on $\QT^1_2 (A)$. 
  This implies that 
  \[
  \lim_{n \to \infty} \sup_{\tau \in \QT^1_2 (A) }  \left\{ d_{\tau} (a) -  \tau  \left(f_{\tfrac{8}{n} } (a)\right) \right\}=0.
  \]
 We use Lemma~\ref{FunCal}(\ref{FunCal.4}) with $8/n$ in place of 
$\ep$ and $1/n$ in place of  $\delta$ to get 
\[
\tau (f_{8/n} (a)) \leq d_{\tau} \left((a-\tfrac{1}{n})_+\right).
\]
We use this to get 
$
 d_{\tau} (a) -  d_\tau \left( (a -\tfrac{1}{n})_+ \right)
 \leq
 d_{\tau} (a) -  \tau \left(f_{\tfrac{8}{n}}(a) \right)$. 
  Thus,
  \[
\lim_{n \to \infty} \sup_{\tau \in \QT^1_2 (A) }  \left\{  d_{\tau} (a) -  d_\tau \left((a -\tfrac{1}{n})_+ \right) \right\}
 \leq
\lim_{n \to \infty} \sup_{\tau \in \QT^1_2 (A) }  \left\{  d_{\tau} (a) -  d_{\tau} \left(f_{\tfrac{8}{n}}(a) \right) \right\}=0.
 \]
 This completes the proof. 
\end{proof}

For a given $a$ in a Cuntz semigroup $S$, we know that 
$\widehat{a}$ is not continuous in general. 
In the following theorem, we turn to the promised connection between 
 continuity of the rank of a full element, the normalized rank ratio function, and  its
oscillation. 
\begin{thm}
\label{ContRank}
Let $A$ be a unital  C*-algebra with $\QT_{2}^1(A)\neq \emptyset$, and let $a\in \cup_{n=1}^{\infty} M_n(A)_+$  be full with $\| a\| = 1$. 
Then the following are equivalent:
\begin{enumerate}
\item
\label{ContRank.a}
$\lim_{n \to \infty} \rho_{[1_A]} \left([a], [(a-\tfrac{1}{n})_+]\right)=1$.
\item
\label{ContRank.b}
$\lim_{n \to \infty} \sup_{\tau \in \QT^1_2 (A) }  \left\{ d_{\tau} (a) -  d_{\tau} ((a-\tfrac{1}{n})_+) \right\}=0$.
\item
\label{ContRank.c}
$\omega(a)=0$.
\item
\label{ContRank.d}
$\widehat{[a]}$ is continuous on $\cF_{[1_A]} ( \Cu (A))$. 
\end{enumerate}
\end{thm}
\begin{proof}
Since $a$ is full and $\lim_{n \to \infty} \left(a-\tfrac{1}{n}\right)_+ =a$, it follows from Lemma~\ref{FuSeq.El.Lem} that there exists
 $n_0 \in \N$ such that 
 $\left(a-\tfrac{1}{n}\right)_+$ is full for all $n\geq n_0$.
 Set  $\eta = \inf_{\tau \in \QT_{2}^{1} (A)} d_\tau (a)$. It follows from Lemma~\ref{InfofFullElem}(\ref{InfofFullElem.c}) that
$\eta >0$. 
 
We prove the implication ``(\ref{ContRank.b})  $\Rightarrow$ (\ref{ContRank.a})". So assume that
\[
\lim_{n \to \infty} \sup_{\tau \in \QT^1_2 (A) } \left\{ d_{\tau} (a) -  d_{\tau} ((a-\tfrac{1}{n})_+) \right\}=0.
\] 
 Let $\ep \in (0, \infty)$. 
 Now set $\delta = \tfrac{1}{2} \cdot \min \left( \eta, \tfrac{\ep \eta}{1+\ep} \right)$.
 It is clear that $ \tfrac{\dt}{\eta - \dt} < \ep$. We choose $n_1 \in \N$ such that $n_1\geq n_0$ and, 
 for all $n \geq n_1$, we have 
 \begin{equation}
 \label{Eq1.2025.03.08}
 d_{\tau} (a) - d_{\tau} 
 \left(
 (a-\tfrac{1}{n})_+ 
 \right) < \delta
 \quad
 \mbox{for all } \tau \in \QT_{2}^{1} (A).
 \end{equation}
Using this, we get,  for all $n \geq n_1$ and all $\tau \in \QT_{2}^{1} (A)$,
\begin{equation*}
0< \eta - \delta < d_{\tau} (a) - \dt < d_{\tau} 
 \left(
 (a-\tfrac{1}{n})_+ 
 \right).
\end{equation*}
This implies that, for all $n \geq n_1$, 
\begin{equation}
\label{Eq2.2025.03.08}
0< 
\left[ d_{\tau} \left( (a-\tfrac{1}{n})_+  \right)\right]^{-1} < \tfrac{1}{\eta - \dt}
\quad
 \mbox{for all } \tau \in \QT_{2}^{1} (A).
\end{equation}
Now, for all $n \geq n_1$ and all $\tau \in \QT_{2}^{1} (A)$, multiply (\ref{Eq1.2025.03.08}) by 
$\left[ d_{\tau} \left( (a-\tfrac{1}{n})_+  \right)\right]^{-1}$, we get
\begin{equation}
\frac{d_{\tau} (a) }{d_{\tau} \left( (a-\tfrac{1}{n})_+  \right)}
\leq
1+ \frac{\dt }{d_{\tau} \left( (a-\tfrac{1}{n})_+  \right)}
\underset{\text{(\ref{Eq2.2025.03.08})}}{<} 
 1+ \frac{\dt }{\eta - \dt} < 1+ \ep. 
\end{equation}
This implies that, for all $n \geq n_1$, we have
\[
 1\leq \sup_{\tau \in \QT^1_2(A)} \left( \frac{d_{\tau} (a) }{d_{\tau} \left( (a-\tfrac{1}{n})_+  \right)} \right) < 1+ \ep.
 \] 
Using Proposition~\ref{Prp.rhoequiavalent}(\ref{Prp.rhoequiavalent.a}),
we get,
for all $n \geq n_1$,
\[
 1\leq \rho_{[1_A]} \left( [a], [(a-\tfrac{1}{n})_+]\right)< 1+ \ep.
 \]
So, (\ref{ContRank.a}) follows. 

We prove the implication ``(\ref{ContRank.a}) $\Rightarrow$ (\ref{ContRank.b})". 
So, suppose 
$
\lim_{n \to \infty} \rho_{[1_A]} \left([a], [(a-\tfrac{1}{n})_+]\right)=1$. 
We choose $k \in \N$ such that $a \in M_k (A)_+$. Thus $d_{\tau} (a) \leq k$ for 
all $\tau \in \QT_{2}^{1} (A)$. Let $\ep>0$. We choose $m_0 \in \N$ such that $m_0 \geq n_0$ and, for all $m \geq m_0$, we have 
$\rho_{[1_A]} \left([a], [(a-\tfrac{1}{n})_+]\right)-1 <\frac{\ep}{2k}$. 
Using Proposition~\ref{Prp.rhoequiavalent}(\ref{Prp.rhoequiavalent.a}), we get, for all $m \geq m_0$ 
and for all $\tau \in \QT_{2}^{1} (A)$,
\[
\frac{d_{\tau} (a) }{d_{\tau} \left( (a-\tfrac{1}{m})_+  \right)} -1 < \frac{\ep}{2k}.\]
Using this and the fact that $d_{\tau} \left( (a-\tfrac{1}{m})_+  \right) \leq d_{\tau} (a)\leq k$, we get, 
for all $m \geq m_0$, 
\[
d_{\tau} (a) - d_{\tau} \left( (a-\tfrac{1}{m})_+  \right) < \frac{\ep}{2k} \cdot d_{\tau} \left( (a-\tfrac{1}{m})_+  \right) < \ep
\quad
\mbox{
for all } \tau \in \QT_{2}^{1} (A).
\]
So (\ref{ContRank.b}) follows. 

Now, 
we prove the implication ``(\ref{ContRank.d})  $\Rightarrow$ (\ref{ContRank.c})".
For $n\in \N$, define $g_n \colon \QT_2^{1} (A) \to [0, \infty)$ by $g_n (\tau) = \tau  \big( f_{\tfrac{1}{n}} (a) \big)$.
Since $g_n \leq g_{n+1}$ and 
$g_n (\tau) \to \widehat{a} (\tau)$, and 
$\widehat{a}$ is continuous, it follows from 
Dini's Theorem that $(g_n)_{n=1}^{\infty}$
 converges uniformly to $\widehat{a}$ on 
 $\QT_2^{1} (A)$. So (\ref{ContRank.c}) follows. 
 
 Finally, we prove the implication``(\ref{ContRank.c})  $\Rightarrow$ (\ref{ContRank.d})".
 Since 
 \[
 \lim_{n \to \infty} \sup_{\tau \in \QT^1_2 (A)} 
\left\{ 
d_{\tau} (a) -  \tau \big( f_{\tfrac{1}{n}} (a) \big) \right\}=0,
\]
 it follows that $(g_n)_{n \in \N}$
converges uniformly to $\widehat{[a]}$ on 
 $\QT_2^{1} (A)$. Using the fact that $g_n$ is continuous for all $n \in\N$, we get $\widehat{a}$ is continuous. 
 
 It follows from Lemma~\ref{Oscillation-Equvallence} that
 (\ref{ContRank.b})  is equivalent (\ref{ContRank.c}). 
 This completes the proof. 
\end{proof}
If the C*-algebra $A$ in the above theorem is residually stably finite, then it follows from Proposition~\ref{Prp.rhoequiavalent}(\ref{Prp.rhoequiavalent.c}) that $\rho_{[1_A]}$ can be replaced with $\rho$. 
\section{The relative radius of comparison and rank ratio function}

The aim of this section is to shed light on the relationships between the relative radius of comparison and rank ratio function. The main question is: Given a sequence of full elements $(x_n)_{n \in \N}$ in a $\Cu$-semi group $S$, what can be determined  about $\lim_{n \to \infty} \rc(S, x_n)$? We try to answer this first. Further, we define a function called the radius of comparison function and its reciprocal. We discuss the main properties of them. 
This function plays an important role in capturing nonclassifiable C*-algebras. 

The following lemma indicate the potential values of relative of comparison of a given $\Cu$-semigroup $S$, if the rank ratio function is zero at a pair of full elements $S$, no matter whether $S$ has  a normalized functionl or not. 
\begin{lem}\label{RkRation0}
Let $S$ be a $\Cu$-semigroup. Let $x, y \in \Full(S)$. 
Suppose that $\rho (x, y) =0$, then: 
\begin{enumerate}
\item\label{RkRation0.a_000}
$\rc(S, y) \leq \rc(S, n \cdot x)$ for all $n \in \N$. 
\item\label{RkRation0.a}
If $\rc (S, x) < \infty$, then $\rc(S, y)=0$.
\item\label{RkRation0.a_0}
If $\rho (y, x)=0$, then 
$\rc(S, y) = \rc(S, x)$. 
\item\label{RkRation0.c}
If  $\rho (y, x)=0$ and $\rc(S, y)<\infty$, then 
$\rc(S, x) = \rc(S, y)=0$.
\end{enumerate}
\end{lem}
\begin{proof}
We prove (\ref{RkRation0.a_000}). 
Since $\rho(x, y)=0$, it follow from  that, for every $n \in \N$, we have $\widehat{x} \leq \frac{1}{n} \cdot \widehat{y}$. 
So, for every $n \in \N$, we have $\lambda(n \cdot x) \leq \lambda (y)$ for all $\lambda \in \cF(S)$. 
Using Proposition~\ref{rcProperties}(\ref{rcProperties.a_0}), we have $\rc (S, y) \leq   \rc(S, n \cdot x)$ for all $n \in \N$. 

Part (\ref{RkRation0.a}) follows immediately from (\ref{RkRation0.a_000}) and Proposition~\ref{rcProperties}(\ref{rcProperties.d}).

We prove (\ref{RkRation0.a_0}). 
Since $\rho(x, y)=\rho(y, x)= 0$, it follows from 
Lemma~\ref{Rkiszero} that $\cF_{x} (S) = \cF_{y} (S) = \emptyset$. 
So $
\Im (\widehat{x}) = \Im (\widehat{y})
= \{0, \infty\}$. This implies that 
$\lambda (x)= \lambda(y)$ all $\lambda \in \cF(S)$. 
Then, by Proposition~\ref{rcProperties}(\ref{rcProperties.b}), we have $\rc (S, x)= \rc (S, y)$.  
 
Part (\ref{RkRation0.c}) follows immediately from (\ref{RkRation0.a}) and (\ref{RkRation0.a_0}).

\end{proof}

In the setting of C*-algebras, the purely infinite ones are examples satisfying Lemma~\ref{RkRation0}(\ref{RkRation0.c}) (see Example~\ref{Exa.PurelyInfinite}). 
\begin{rmk}\label{Fulfunctional}
Let $S$ be a $\Cu$-semigroup. Let $x, y \in \Fu (S)$ with $x\leq y$. Let $\lambda \in \cF_y (S)$. Since $x$ is full and $x\le y$,
we have $0<\lambda(x)< \infty$. So, 
$\tfrac{1}{\lambda (x)} \cdot \lambda \in \cF_x(S)$, and,  therefore, $\cF_x(S)\neq \emptyset$. 
\end{rmk}

Let $S$ be a $\Cu$-semigroup and let $x, y \in \Fu (S)$ with 
$x \leq y$. It immediate follows from Proposition~\ref{rcProperties}(\ref{rcProperties.a_1}) that $\rc (S, y) \leq \rc (S, x)$. The fist part of the following theorem provides a better lower bound for
 $\rc (S, x)$.
 The second part of the following theorem provides an upper bound for $\rc (S, x)$, which is based on the rank ratio function at $(y ,x)$. 
\begin{thm}\label{Rkrc}
Let $S$ be a $\Cu$-semigroup. Let $x, y \in \Full(S)$ with $x\leq y$. Let $\cF_y (S) \neq \emptyset$. Then 
\begin{enumerate}
\item \label{Rkrc.a}
$\frac{1}{\rho(x, y)} \cdot \rc (S, y) \leq \rc(S, x)$.
\item\label{Rkrc.b}
If further $\rho(y, x)< \infty$, 
then
\[
\frac{1}{\rho(x, y)} \cdot \rc (S, y) \leq \rc(S, x)
\leq 
\rho(y, x) \cdot \rc(S, y). 
\]
\end{enumerate}
\end{thm}
\begin{proof}
We prove (\ref{Rkrc.a}). 
Since $x\leq y$ if follows from Proposition~\ref{rcProperties}(\ref{rcProperties.a_1})
that $\rc (S, y) \leq \rc (S, x)$.
Since $\cF_y(S)\neq \emptyset$ and $x \leq y$, it follows from Remark~\ref{Fulfunctional} and 
 Lemma~\ref{RkRatioProperty}(\ref{RkRatioProperty.b}) that $0<\rho(x, y)\leq 1$. So $\tfrac{1}{\rho(x, y)}\geq 1$. 
 Now, we prove $\rc(S, y) \leq \rho(x, y) \cdot \rc (S, x)$.
 If $\rc (S, x) = \infty$, then the inequality is trivial, considering the convention $\gamma \cdot \infty = \infty$ for $\gamma \in (0, \infty)$. 
 So, suppose that $\rc (S, x)< \infty$. 
 Assume that  $S$ has $\eta$-comparison relative to $x$.
 Since $\rho(x, y)\leq 1$, for every $\ep>0$, there exists $r>0$ in $\cR(\widehat{x}, \widehat{y})$ such that 
 $r< \rho(x, y) + \ep$. 
 So, we have 
 $
 \widehat{x} \leq r \cdot \widehat{y}
 \leq 
 (\rho(x, y) + \ep)\cdot \widehat{y}$. 
 Using this and Proposition~\ref{rcProperties}(\ref{rcProperties.a_0}), for every $\ep>0$, we get
 $\rc (S, y) \leq  (\rho(x, y) + \ep)\cdot \rc(S, x)$. 
 Thus, 
  $\rc (S, y) \leq  \rho(x, y) \cdot \rc (S, x)$.
  
 Now,  we prove (\ref{Rkrc.b}).
 The first part of the inequality follows from (\ref{Rkrc.a}). So, we only prove that 
$\rc(S, x)
\leq 
\rho(y, x) \cdot \rc(S, y)$. 
 If $\rc (S, y) = \infty$, then the equation is trivial, using the fact that $\rc (S, y)\leq \rc(S, x)$ and considering the convention $\gamma \cdot \infty = \infty$ for $\gamma \in (0, \infty)$. 
 Since $\cF_y (S) \neq \emptyset$, the fact that $x\leq y$, and $x \in \Fu (S)$, we have 
 $\cF_x(S) \neq \emptyset$ by Remark~\ref{Fulfunctional}. Then, 
  Lemma~\ref{RkRatioProperty}(\ref{RkRatioProperty.c}), we have $\rho(y, x) \geq 1$. 
  Since $\rho(y, x)< \infty$, for every $\ep>0$, there exists $r>0$ in $\cR(\widehat{y}, \widehat{x})$ such that 
  $r< \rho(y, x) + \ep$. 
  So, we have
$
 \widehat{y} \leq r \cdot \widehat{x}
 \leq 
 (\rho(y, x) + \ep)\cdot \widehat{x}$. 
 Using this and Proposition~\ref{rcProperties}(\ref{rcProperties.a_0}), for every $\ep>0$, we get
 $
\frac{1}{\rho(y, x) + \ep} \cdot \rc (S, x) \leq \rc(S, y)$. 
 Thus, $\frac{1}{\rho(y, x)} \cdot \rc (S, x) \leq \rc(S, y)$. 
 This completes the proof.  
  \end{proof}
  
  See Example~\ref{nonunitalrcrankratio}, Example~\ref{RCPand1}, and Example~\ref{rceveryGroup} for many concrete examples satisfying Theorem~\ref{Rkrc}. 

The following theorem determines an upper bound and a lower bound for the limit of the radii of comparison of a 
given $\Cu$-semigroup relative to an increasing sequence in the $\Cu$-semigroup based on the rank ratio structure of the sequence. 
\begin{thm}\label{MainThmForRCLim}
Let $S$ be a $\Cu$-semigroup. Let $y$ and a sequence $(y_n)_{n=1}^{\infty}$ be in $\Full(S)$ with $y_n\leq y_{n+1} \leq y$ for all $n \in \N$. Assume that $\cF_y (S) \neq \emptyset$ and
$\lim_{n \to \infty} \rho(y, y_n) < \infty$. 
Then 
\[
\frac{1}{\lim_{n \to \infty} \rho(y_n, y)} \cdot  \rc (S, y) \leq
 \lim_{n \to \infty} \rc(S, y_n)
\leq 
\lim_{n \to \infty} \rho(y, y_n) \cdot \rc(S, y). 
\]
\end{thm}
\begin{proof}
Since $\cF_y(S)\neq \emptyset$ and 
$y_n \precsim y$ for all $n \in \N$, it follows from Remark~\ref{Fulfunctional}
that 
$\cF_{y_n} (S) \neq \emptyset$ for all $n \in \N$.  
Since $y_1 \leq y_2 \leq \ldots \leq y_n \leq y_{n+1} \leq \ldots y$, it follows from
Proposition~\ref{rcProperties.b}(\ref{rcProperties.b}) that 
\[
\rc(S, y) \leq \ldots \leq \rc (S, y_{n+1})
\leq \rc (S, y_n) \leq \ldots \leq \rc (S, y_2) \leq \rc (S, y_1). 
\]
Case 1: if $\rc (S, y)=\infty$, then $\rc(S, y_n)=\infty$ for all $n \in \N$. 
Therefore, 
$\lim_{n \to \infty} \rc (S, y_n) = \rc (S, y) =\infty$ and 
the inequality is trivially true, 
considering the convention $\gamma \cdot \infty = \infty$ for $\gamma \in (0, \infty)$. 

Case 2: Suppose $\rc (S, y)< \infty$. 
It follows from Theorem~\ref{LimofRank}(\ref{LimofRank.a}) that 
the sequence  
$(\rho(y_n, y))_{n=1}^{\infty}$ in an increasing sequence in $\mathbb{R}$ and 
 $0<\lim_{n \to \infty} \rho(y_n, y) \leq 1$.
 Also, by Theorem~\ref{LimofRank}(\ref{LimofRank.b}), we know that
 $(\rho(y, y_n))_{n=1}^{\infty}$ in an decreasing sequence bounded below by $1$. Since 
$\lim_{n \to \infty} \rho(y, y_n) < \infty$, 
we choose $N \in \N$ such that, for all $n \geq N$, we have $\rho(y, y_n) <\infty$. Using this and Theorem~\ref{Rkrc}(\ref{Rkrc.b}), we 
get, for all $n\geq N$, 
\begin{equation}\label{EQ5.2025.03.12}
\frac{1}{\rho(y_n, y)} \cdot \rc (S, y) \leq \rc(S, y_n)
\leq 
\rho(y, y_n) \cdot \rc(S, y). 
\end{equation}
This implies that $\rc(S, y_n)<\infty$ for $n \geq N$, and therefore, the limit of the sequence $(\rc(A, y_n))_{n=N}^{\infty}$ exists and is in $\mathbb{R}$. So, by taking the 
limit of each sides of (\ref{EQ5.2025.03.12}) as $n \to \infty$, the results follows. 
\end{proof}
See Example~\ref{EXlimiISZero} and Example~\ref{RankratioMeasureTheo} for more concrete examples satisfying Theorem~\ref{MainThmForRCLim}.  
Also, Theorem~\ref{MainThmForRCLim} is not valid if the sequence is not increasing (see Example~\ref{EXlimiISZero}(\ref{EXlimiISZero.b}))

As you have seen in the poof of the theorem above, when $\lim_{n \to \infty} \rho(y, y_n) < \infty$ and $\rc (S, y)<\infty$, then
$\lim_{n \to \infty} \rc(S, y_n)$ exits and is in $\mathbb{R}$.  This this limit is not
 equal to $\rc (S, y)$ in general. However, if 
 $\lim_{n \to \infty} \rho(y, y_n) =1$ this is true and we have the following corollary.
\begin{cor}\label{rclimrho1}
Let $S$ be a $\Cu$-semigroup. Let $y$ and a sequence $(y_n)_{n=1}^{\infty}$ be in $\Full(S)$ with $y_n\leq y_{n+1} \leq y$  for all $n \in \N$. Assume that $\cF_y (S) \neq \emptyset$ and 
$\lim_{n \to \infty} \rho(y, y_n) =1$. 
Then 
$
 \lim_{n \to \infty} \rc(S, y_n)
= \rc(S, y)$. 
\end{cor}
In Proposition~\ref{ExtremeInfiniteRho} and Proposition~\ref{ExtremeZeroRho}, we focus on extreme cases where $\lim_{n \to \infty} \rho(y, y_n) = \infty$ or $\lim_{n \to \infty} \rho(y, y_n) = 0$. 
\begin{prp}\label{ExtremeInfiniteRho}
Let $S$ be a $\Cu$-semigroup. Let $y$ and a sequence $(y_n)_{n=1}^{\infty}$ be in $\Full(S)$ with $y_n\leq y_{n+1} \leq y$ for all $n \in \N$. Assume that $\cF_y (S) = \emptyset$ and
$\lim_{n \to \infty} \rho(y, y_n) = \infty$. 
Then 
\begin{enumerate}
\item\label{ExtremeInfiniteRho.a}
$\cF_{y_n} (S) \neq \emptyset$ for all $n \in \N$. 
\item\label{ExtremeInfiniteRho.b}
If $\rc (S, y_m)< \infty$ for some $m \in \N$, then
\[
 \lim_{n \to \infty} \rc(S, n \cdot y_n)
= \rc(S, y)=0. 
\]
\end{enumerate}
\end{prp}
\begin{proof}
To prove (\ref{ExtremeInfiniteRho.a}), suppose that 
$\rho (y, y_m) < \infty$ for some $m \in \N$. 
It follows from Theorem~\ref{LimofRank}(\ref{LimofRank.b}) that
the sequence  
$(\rho(y, y_n))_{n=1}^{\infty}$ in an decreasing  sequence in $[0, \infty]$. 
This implies that that $\lim_{n \to \infty} \rho(y, y_n) <\infty$ which is a contradiction. 
So $\rho (y, y_n) = \infty$ for all $n \in \N$. 
Now, using Lemma~\ref{Rkiszero}, we get 
$\cF_{y_n} (S) \neq \emptyset$. 

We prove (\ref{ExtremeInfiniteRho.b}).
Using the fact that $\cF_{y} (S) = \emptyset$ together with 
Lemma~\ref{Rkiszero}, we get $\rho(y_n, y)= 0$ for all $n \in \N$. Now it follows from Lemma~\ref{RkRation0}(\ref{RkRation0.a_000}) that 
\[
\rc (S, y) \leq \rc (S, n\cdot y_n) 
\quad
\mbox{for all } n \in \N.
\]
Using this at the fist step, using the fact that $n \cdot y_m \leq n \cdot y_n$ for all $n \geq m$ together with Proposition~\ref{rcProperties}(\ref{rcProperties.a_1}) at the second step, and using Proposition~\ref{rcProperties}(\ref{rcProperties.d}) at the third step, we get
\[
\rc(S, y) \leq \rc (S, n \cdot y_n) \leq 
\rc (S, n \cdot y_m) = \frac{1}{n} \cdot \rc (S, y_m). 
\]
Using this and the fact that $\rc (S, y_m)< \infty$, we get
\[
\lim_{n \to \infty} \rc(S, n \cdot y_n)
= \rc(S, y)=0,
\] 
as desired.
\end{proof}
We refer to Example~\ref{AlgebraciExample} for 
an example 
that $\lim_{n \to \infty} \rho(y, y_n) = \infty$ and 
that satisfies the above proposition.

\begin{prp}\label{ExtremeZeroRho}
Let $S$ be a $\Cu$-semigroup. Let $y$ and a sequence $(y_n)_{n=1}^{\infty}$ be in $\Full(S)$ with $y_n\leq y_{n+1} \leq y$ for all $n \in \N$. If 
$\lim_{n \to \infty} \rho(y, y_n) = 0$, then: 
\begin{enumerate}
\item\label{ExtremeZeroRho.a}
There is $m \in \N$ such that
$\cF_y(S) = \cF_{y_n}(S) = \emptyset$ 
for all $n \geq m$.
\item\label{ExtremeZeroRho.b}
$
 \lim_{n \to \infty} \rc(S,  y_n)
= \rc (S, y). 
$
\item\label{ExtremeZeroRho.c}
If  $\rc (S, y)< \infty$, then
$ \lim_{n \to \infty} \rc(S,  y_n)
= \rc (S, y)=0$.
\end{enumerate}
\end{prp}
\begin{proof}
We prove (\ref{ExtremeZeroRho.a}). 
We use $\lim_{n \to \infty} \rho(y, y_n) = 0$ to
choose $m \in \N$ such that 
\begin{equation}\label{Eq1.2025.04.08}
\rho(y, y_{m}) < \frac{1}{2m}. 
\end{equation}
 Since $\rho(y, y_{m}) < \infty$, we choose 
$r >0$ in $\cR(\widehat{y}, \widehat{y_m})$ such that 
\begin{equation}
\label{Eq222.2025.04.08}
r < \rho (y, y_m) + \frac{1}{2 m}.
\end{equation}
 Then
\[
\lambda (y) \leq r \cdot \lambda (y_m) 
\overset{\text{(\ref{Eq1.2025.04.08})}}{\underset{\text{(\ref{Eq222.2025.04.08})}}{\leq}}
 \frac{1}{m} \cdot 
\lambda (y_m)
\quad
\mbox{ for all }
\lambda \in \cF(S). 
\]
This implies that 
$m \cdot \lambda (y) \leq \lambda (y_m)$
for all $\lambda \in \cF(S)$.
Since $y_m \leq y$, we have 
$m \cdot \lambda ( y) \leq \lambda (y_m)\leq \lambda (y)$
for all $\lambda \in \cF(S)$.
So $\lambda (y_m), \lambda (y) \in \{0, \infty\}$
for all $\lambda \in \cF(S)$.
This implies that $\cF_y(S)= \cF_{y_m}(S)= \emptyset$
and, therefore, 
by remark~\ref{Fulfunctional}, we have 
$\cF_{y_n}(S)= \emptyset$ for all $n \geq m$. 

We prove (\ref{ExtremeZeroRho.b}).
Using Lemma~\ref{Rkiszero} and (\ref{ExtremeZeroRho.a}), we have $\rho (y, y_m)=\rho(y_m, y)=0$.
Now, we use this and  Lemma~\ref{RkRation0}(\ref{RkRation0.a_0}) to get
$\rc (S, y)=\rc(S, y_m)$. 
Since  $(\rc (S, y_n))_{n=1}^{\infty}$ is an decreasing sequence, we have 
 $\lim_{n \to \infty} \rc(S,  y_n)
= \rc (S, y)$. 

Part (\ref{ExtremeZeroRho.c}) follows from  Lemma~\ref{RkRation0}(\ref{RkRation0.c}), 
 (\ref{ExtremeZeroRho.b}), and the fact that $\rc (S, y)<\infty$. 
\end{proof}

The following functions play a key role in studying non-classifiable C*-algebra. 
\begin{dfn}\label{Rciprical}
Let $S$ be a $\Cu$-semigroup with $\Full (S)\neq \emptyset$. 
\begin{enumerate}
\item 
The \emph{radius of comparison function associated to $S$}, denoted by $\rc (S, \cdot)$, is the function
$\rc (S, \cdot) \colon \Fu (S) \to 
[0, \infty]$ given by $x \mapsto \rc (S, x)$.
\item 
The 
\emph{reciprocal radius of comparison function associated to $S$}, denoted by $\Irc_S$, is defined as follows: 
\[
 \Irc_S \colon \Full(S) \to [0, \infty], 
 \quad
 \quad
 \Irc_S (x) = \frac{1}{\rc(S, x)},\]
with the  convention that
$\frac{1}{\infty}=0$ and $\frac{1}{0}=\infty$.
\end{enumerate}
\end{dfn}
We consider the reciprocal of $\rc (S, \cdot)$ only because of the fact that
$\Irc_S$
is homogeneous and preserves the order. 
Also, the letter ``I"  in the notation $\Irc_S$ stands for 
``(multipilicative) inverse". 

The following proposition provides some important facts about the reciprocal radius of comparison function. 
\begin{prp}\label{IrcProperties}
Let $S$ be a $\Cu$-semigroup with $\Full (S)\neq \emptyset$. Then:
\begin{enumerate}
\item\label{IrcProperties.a}
$\Irc_S$ preserves the order.
\item\label{IrcProperties.b}
$\Irc_S (n. x)= n \cdot \Irc_S(x)$ for all $x \in \Full(S)$ and $n\in \N$.
\item\label{IrcProperties.c}
Let 
$(x_n)_{n \in \N}$ be an increasing sequence in $\Full(S)$ with $\sup_n x_n =x$ 
for $x \in S$.  If $\cF_x(S) \neq \emptyset$ and 
$\lim_{n \to \infty} \rho(x, x_n)=1$, 
then
$\sup_n \Irc_S ( x_n) = \Irc_S (x)$. 
\end{enumerate}
\end{prp}
\begin{proof}
To prove (\ref{IrcProperties.a}), let $x, y \in \Fu (S)$ with $x\leq y$. Then, by Proposition~\ref{rcProperties}(\ref{rcProperties.a_1}), we have 
\begin{equation}\label{EQ10.2025.03.12}
\rc (S, y) \leq \rc (S, x).
\end{equation}

Case 1: If $0<\rc (S, y) \leq \rc (S, x)< \infty$, then $\Irc_S(x)=\frac{1}{\rc (S, x)} \leq \frac{1}{\rc (S, y)}=  \Irc(y)$. 

Case 2: If $\rc (S, x) \in \{0, \infty\}$ or 
$\rc (S, y) \in \{0, \infty\}$,
then, the definition of $\Irc_S$ and (\ref{EQ10.2025.03.12}), it is obvious that $\Irc_S (x) \leq \Irc_{S} (y)$. 

Part (\ref{IrcProperties.b}) follows from Proposition~\ref{rcProperties}(\ref{rcProperties.c}) and the definition of 
$\Irc_S$. 

We prove (\ref{IrcProperties.c}).
Since $\lim_{n \to \infty} \rho(x, x_n)=1$, it follows from Corollary~\ref{rclimrho1} that
$ \lim_{n \to \infty} \rc(S, x_n)
= \rc(S, x)$. 
\begin{itemize}
\item[
Case 1] If $0<\rc(S, x)<\infty$, then
\[
\sup_n \Irc_S (x_n) = 
\lim_{n \to \infty} \frac{1}{\rc(S, x_n)}
=
\lim_{n \to \infty} \frac{1}{\rc(S, x)}
=
\Irc_S (x). 
\]
\item[Case 2] If $\rc (S, x)= \infty$, then, by Proposition~\ref{rcProperties}(\ref{rcProperties.a_1}), $\rc (S, x_n)= \infty$. So we have $\Irc_S (x) = \sup_n \Irc_S (x_n) = 0$. 
\item[Case 3] Suppose $\rc (S, x)=0$. 
\begin{itemize}
\item[Case 3.a]
If $\rc (S, x_n)>0$ for $n\in \N$, then
$\lim_{n \to \infty} \frac{1}{\rc (S, x_n)}= \infty$. So we have
$\sup_n \Irc_S (x_n)= \Irc_S (x)=\infty$.
\item[Case 3.b]
If there exist $N\in \N$ such that $\rc (S, x_N)=0$, then $\rc (S, x_n)= \rc(S, x)= 0$ for all $n \geq N$. Therefore, 
$\sup_n \Irc_S (x_n)= \Irc_S (x)=\infty$. 
\end{itemize}
\end{itemize}
This completes the proof. 
\end{proof}

One may ask what the range of $\rc (S, \cdot)$ or $\Irc_S$ is for a given  $\Cu$-semigroup $S$. 
In Example~\ref{Exa.PurelyInfinite}(\ref{Exa.PurelyInfinite.d}), it is shown that
for any purely infinite C*-algebra $A$, the range of
$\rc(\Cu (A), \cdot)$ is $\{0\}$ and, therefore, the range of  
$\Irc_{\Cu(A)}$ is $\{\infty\}$.  In the following lemma, we obtain the range of the radius of comparison function associated to classifiable C*-algebras. 
\begin{lem}\label{IrcClassi}
Let $A$ be a simple unital stably finite C*-algebra.
Suppose that $\rc (\Cu (A), [1_A])=0$. 
Then $\rc\Big({\Cu(A)},  \Fu ( \Cu (A) ) \Big)= \{0\}$.
\end{lem}
\begin{proof}
Since $A$ is simple and every nonzero element is full, 
it suffices to show that, 
$\rc (\Cu (A), [a])=0$ for every $a \in (A\otimes \cK)_+ \setminus \{0\}$. 
Let $a \in (A\otimes \cK)_+ \setminus \{0\}$. 
We may assume $\| a \|=1$. Choose $\ep >0$ such $(a-\ep)_+ \neq 0$. Now, we use  Lemma~1.9 of \cite{Ph14} to choose $n \in \N$ and $b \in M_n (A)_+$ such that $(a-\ep)_+ \sim b$. Since $b \leq \| b\| \cdot 1_{M_n (A)}$, it follows from Proposition~\ref{Prp.rhoequiavalent}(\ref{Prp.rhoequiavalent.c}) 
that $\rho(n \cdot [1_A], [b]) <\infty$. 
Now, using Proposition~\ref{rcProperties}(\ref{rcProperties.a_1}) at the first step, 
using Theorem~\ref{Rkrc}(\ref{Rkrc.b}) at the second step, and using  Proposition~\ref{rcProperties}(\ref{rcProperties.d}) and $\rc (\Cu (A), [1_A])=0$ at the last step step, we get 
\[
\rc \left(\Cu (A), [(a-\ep)_+] \right)= \rc \left(\Cu (A), [b] \right)
\leq 
\rho(n \cdot [1_A], [b]) \cdot \rc \left(\Cu (A), n\cdot [1_A] \right)
=
0.
\]
 Then, by Proposition~\ref{rcProperties}(\ref{rcProperties.a_1}) and the fact that $(a-\ep)_+ \precsim a$, we have 
$\rc (\Cu (A), [a])=0$,  as desired. 
\end{proof}

In the above lemma, it is clear by the definition~ that $\Irc_{\Cu (A)} (\Cu (A)) = \{\infty\}$.

 In general, the rank of an element in a $\Cu$-semigroup have been turned out to be very important, particularly, its association to ``rank problem", i.e.,
the problem of representing every strictly positive, lower semicontinuous, affine function on the functional space of a Cuntz semigroup $S$ as the rank of an element in $S$, which  remains open in the general setting.
This problem is first studied in \cite{DaTo10} and then later in \cite{APRT22, ERS11, Rob13, Thi20} for certain classes of C*-algebras.  
Combing our result here and 
a positive solution to the rank problem for C*-algebras with stable rank one (see \cite{APRT22, Thi20}), we have the following theorem. 

\begin{thm}\label{ThmManufactureRc}
Let $A$ be a unital separable residually stably finite C*-algebra with stable rank one and with no nonzero, finite dimensional quotients. Let $r \in (0, \infty)$. Suppose  that $\rc (\Cu (A), [1_A])=r$. Then 
\[
\rc \big(\Cu (A), \Fu (\Cu (A))\big)= [0, \infty).
\] 
\end{thm}

\begin{proof}
It is enough to show that 
for any $\eta \in [0, \infty)$, there exist a full element $a \in (A\otimes \cK)_+$
such that $\rc (\Cu (A), [a])= \eta$. 
\begin{itemize}
\item[Case 1]
Assume that $\eta \in (0, \infty)$. 
 Choose $n \in \N$ such that $r \eta^{-1}  < n$ and define 
\[
f \colon \cF_{[1_{A}]} (\Cu (A)) \to (0, \infty], \quad
f(\lambda) = r \eta^{-1}.\]
So $f$ is an affine function
 satisfying the hypothesis of Theorem~7.14 of \cite{APRT22}. 
 Then we choose $a \in M_n (A)_+$ such that 
 \begin{equation}
 \label{EQ1.2025.05.02}
 \lambda ([a]) = r \eta^{-1}
 \quad
 \mbox{ for all } 
\quad 
 \lambda \in \cF_{[1_{A}]} (\Cu (A)).
 \end{equation}
 It follows from Lemma~\ref{ResiduallyAndFullEquivalent} that $a$ is full. 
Now,  using (\ref{EQ1.2025.05.02}) and Proposition~\ref{Prp.rhoequiavalent}(\ref{Prp.rhoequiavalent.c}), we get
 \begin{equation}\label{Eq1.2025.03.29}
\rho ([a], n \cdot [1_{A}]) = 
\sup_{\lambda \in \cF_{[1_{A}]} (\Cu (A)) } 
\left(
 \frac{\lambda ([a])}{\lambda (n \cdot [1_{A}])} \right)=
 \frac{r}{n \eta}, 
 \end{equation}
and
 \begin{equation}\label{Eq2.2025.03.29}
\rho ( n \cdot [1_{A}], [a]) = \sup_{\lambda \in \cF_{[1_{A}]} (\Cu (A)) } 
\left( 
\frac{\lambda (n \cdot [1_{A}])}{\lambda ([a])} \right)= \frac{n\eta}{ r}.
\end{equation}
Therefore, using (\ref{Eq1.2025.03.29}), (\ref{Eq2.2025.03.29}),  and Theorem~\ref{Rkrc}(\ref{Rkrc.b}), 
\begin{equation}
\frac{n \eta}{r} \cdot \rc (\Cu (A), n \cdot [1_{A}]) 
\leq
 \rc(\Cu (A), [a])
\leq 
\frac{n\eta}{ r} \cdot \rc(\Cu (A), n \cdot [1_{A}]). 
\end{equation}
Using this at the first step, using
Proposition~\ref{rcProperties}(\ref{rcProperties.d}) at the second step, and  
using the fact that $\rc(\Cu (A), [1_A])=r$ at the last step, we get 
\[
\rc(\Cu (A), [a])=
\frac{n\eta}{ r} \cdot \rc(\Cu (A), n \cdot [1_{A}])
=
\frac{\eta}{ r} \cdot \rc(\Cu (A), [1_{A}])
=
 \eta. 
\]
\item[Case 2]
 Suppose that $\eta= \infty$. Set $y = \infty \cdot [1_A]$. 
 Then, by Proposition~\ref{rcProperties}(\ref{rcProperties.c}), we have $\rc (\Cu (A), y)=0$. 
\end{itemize} 
 Now,  combining Case 1 and Case 2, the result follows. 
\end{proof}

In proof of the above theorem, since $a \in M_n (A)_+$ is full, it follows from
Remark~\ref{RmkRCHeredit} that 
 $\rc (\Cu \big(\overline{a (A \otimes \cK) a}\big), [a])= \rc (\Cu (A), [a])$.
 Therefore, 
 Theorem~\ref{ThmManufactureRc} in fact tells us that for any $r\in  [0, \infty)$, there exists 
 a full elements in $a \in (A \otimes \cK)_+$
 such  $\rc (\Cu \big(\overline{a (A \otimes \cK) a}\big), [a])=r$.

In the statement of the above theorem, if  
$\rc (\Cu (A), [1_A])=\infty$,  then, by Proposition~\ref{rcProperties}(\ref{rcProperties.a_1}), we have
$
\rc \big(\Cu (A), \Fu ( \W (A) )\big)= \{\infty\}$. 
This is true even if $A$ does not have stably rank one. 

\begin{ntn}
Let $A$ be a C*-algebra, let $G$ be a discrete group, and let $\alpha \colon G \to \Aut (A)$ be an action of $G$ on $A$. We define
\[
\iota \colon A \to A\rtimes_{\alpha} G
\]
to be the inclusion maps. 
 By abuse of notation, we often denote the amplification of this map to matrices or to the stabilization also by the same symbol. 
\end{ntn}
Combing Corollary~5.21 in \cite{ASV23} and 
Theorem~\ref{ThmManufactureRc}, we get the following theorem. 
\begin{prp}\label{rceveryGroup}
For every $m \in \Ntwo$ and every $\eta \in (0, \frac{1}{m})$, there exist
finite group 
$G$ with $|G|=m$, a
simple unital separable AH algebra
 $A$ with stable rank one, and an outer group action $\alpha \colon G \to \Aut (A)$ such that:
 \begin{enumerate}
 \item \label{rceveryGroup.a}
$\rho ([1_A], [p])= m$ and 
$\rho ([p], [1_A])=\frac{1}{m}$. 
 \item \label{rceveryGroup.b}
 $\eta =\rc\big(\Cu (A\rtimes_\alpha G), [\iota (1_A)] \big) =\rc \big(\Cu (A), [1_A]\big)  = \frac{1}{m} \cdot \rc \big(\Cu (A), [p]\big)$. 
 \item\label{rceveryGroup.c}
  Both the rang of 
 $\rc( \Cu (A), \cdot )$ and the range $\rc (\Cu(A\rtimes_\alpha G), \cdot)$
 are $[0, \infty)$. 
 \end{enumerate}
\end{prp}
\begin{proof}
Let $m \in \Ntwo$. Choose a finite group $G$ with 
$m = |G|$. Let $\eta \in (0, \tfrac{1}{m})$. 
Let $A$ be the simple separable unital AH algebra with stable rank one as in Construction~5.8 of \cite{ASV23} and let $\alpha \colon G \to \Aut (A)$ be the pointwise outter action as in Construction~5.8 of \cite{ASV23}. 
By Theorem~5.15 of \cite{ASV23} and Theorem~5.20 of \cite{ASV23}, we have 
\begin{equation}\label{Eq.1_0.2025.04.07}
\Rrc(A\rtimes_\alpha G, \iota  (1_A) ) =\Rrc (A, 1_A) = \eta. 
\end{equation}
Adopt the assumptions and formulas of Notation~5.16 and Construction~5.8 of \cite{ASV23}. 
For every $n \in \N$, set  
$p_n = \Gamma_{n, 0} (p)$ and 
$p_\infty = \Gamma_{\infty, 0} (p)$ 
as in the Notation~5.16 of \cite{ASV23}.
Using Lemma~5.17(2) of \cite{ASV23} and  the Krein–Milman theorem, we get, for all $\tau \in \T^1(A)$, 
\begin{equation}\label{Eq1.2025.03.19}
d_\tau (p_\infty) = \tau (p_\infty)
=
(\tau \circ \Gamma_{\infty, n}) (p_n)= \frac{1}{|G|}. 
\end{equation}
 By (\ref{Eq1.2025.03.19}), the fact that $A$ is residually stably finite, and  Proposition~\ref{Prp.rhoequiavalent}(\ref{Prp.rhoequiavalent.c}), we have 
\begin{equation}\label{Eq2.2025.04.07}
\rho ([1_A], [p])= \rho_{[1_A]} ([1_A], [p]) =|G|
\
\mbox{ and }
\ 
\rho ([p], [1_A])=\rho_{[1_A]} ([1_A], [p])=\frac{1}{|G|}. 
\end{equation}
So
part (\ref{rceveryGroup.a}) follows. 

For  part (\ref{rceveryGroup.b}), we use (\ref{rceveryGroup.a}), 
 Theorem~\ref{CorRrc=rc},  Theorem~\ref{Rkrc}(\ref{Rkrc.b}) with $[p]$ in place of $x$ and $[1_A]$ in place of $y$,  (\ref{Eq.1_0.2025.04.07}), and (\ref{Eq2.2025.04.07})
  to get 
$
|G| \cdot \eta \leq \rc (\Cu (A), [p]) \leq |G| \cdot \eta.
$
This implies that  $\rc (\Cu (A), [p])=|G| \cdot \eta$. 
Using this and (\ref{Eq.1_0.2025.04.07}), Theorem~\ref{CorRrc=rc}, part (\ref{rceveryGroup.b}) follows. 

Part (\ref{rceveryGroup.c}) follows from Theorem~\ref{ThmManufactureRc} and the fact that
$\rc\big(\Cu (A\rtimes_\alpha G), [\iota  (1_A)] \big) =\rc \big(\Cu (A), [1_A]\big)$. 

\end{proof}

In Proposition~\ref{rceveryGroup}(\ref{rceveryGroup.c}), both the rang of 
 $\Irc_{\Cu (A)}$ and the range $\Irc_{\Cu(A\rtimes_\alpha G)}$
 are $(0, \infty]$ by Definition~\ref{Rciprical}.

As an application of the above proposition, we have the following theorem. 
\begin{thm}\label{NonIsoClassofC*}
For every nontrivial finite group $G$, 
there are 
an uncountable family of pairwise nonisomorphic 
unital simple separable AH algebras $\{A_{\gamma} 
 \}_{\gamma \in \Lambda}$, 
 and outter group actions 
$\{\alpha_{\gamma} \colon  G \to A_{\gamma} \}_{{\gamma \in \Lambda}}$,
where $\Lambda$ is an uncountable subset of $(0, \infty)$, 
such that all $A_{\gamma}$ have stable rank one and  the range of both 
$\rc \big( \Cu (A_{\gamma}), \cdot \big)$
and 
$\rc \Big( \Cu (A_{\gamma} \rtimes_{\gamma} G), \cdot \Big)$
are $[0, \infty)$ for all $\gamma \in \Lambda$.
\end{thm}

\begin{proof}
Let $G$ be a finite group with $|G|\geq 2$. 
Choose any pair of distinct real numbers $\mu, \nu \in (0, \tfrac{1}{|G|})$. By Proposition~\ref{rceveryGroup},  there are 
unital separable simple unital 
$A_{\mu}$ and $A_{\nu}$ with stable rank one and outer actions 
$\alpha_{\mu} \colon G \to \Aut (A_\mu)$
and $\alpha_{\nu} \colon G \to \Aut (A_\nu)$
such that:
\begin{enumerate}
\item
$\rc (\Cu (A_\mu), [1_{A_{\mu}}])
=
\rc (\Cu (A_\mu \rtimes G), [1_{A_{\mu}}])=\mu$.
\item
$\rc (\Cu (A_\nu), [1_{A_{\nu}}])
=
\rc (\Cu (A_\nu \rtimes G), [1_{A_{\nu}}])=\nu$.
\end{enumerate}
So it is clear that 
$A_\mu \not\cong A_\nu$ and 
$A_\mu \rtimes_{\alpha_\mu} G \not\cong A_\nu \rtimes_{\alpha_\nu} G$. Further, we know 
that 
\[
\rc \Big( \Cu (A_{\mu}), \ \Fu (\Cu (A_{\mu})) \Big)= 
\rc \Big( \Cu (A_{\mu} \rtimes_{\alpha_\mu} G), \ 
\Fu ( \Cu (A_{\mu} \rtimes_{\alpha_\mu} G))\Big) =[0, \infty),
\]  
and 
  \[
\rc \Big( \Cu (A_{\nu}), \ \big(\Fu (\Cu (A_{\nu})) \Big)= 
\rc \Big( \Cu (A_{\nu} \rtimes_{\alpha_\nu} G), \ 
\big(\Fu ( \Cu (A_{\nu} \rtimes_{\alpha_\nu} G))\Big) =[0, \infty).
\] 
This complete the proof. 
\end{proof}
In the above theorem, we can replace the group $G$ by the group $\mathbb{Z}^{d}$ for different  $d \in \N$, using the result of \cite{AHS25}.

Considering the above proposition, one may ask 
if the family of the C*-algebras $\{A_{\gamma} 
 \}_{\gamma \in \Lambda}$ can be arranged in a way which all of them  have the same Elliott invariant as well. 
 The following proposition answers this question by combing
  Theorem~\ref{ThmManufactureRc} and Theorem~3.11 of 
\cite{HP25}. 
\begin{prp}\label{IrcNonClassi}
There is an
uncountable family of pairwise nonisomorphic 
unital simple separable AH algebras
such that all have stable rank one, all 
possess the  same Elliott invariant, and  
the range of all the radius of comparison functions associated to their Cuntz semigroups is $[0, \infty)$. 
\end{prp}
\begin{proof}
Let $r \in (0, \infty)$. We use Theorem~3.11 of 
\cite{HP25} to choose an uncountable family of simple separable unital AH algebra $\{ A_{\gamma} \colon \gamma \in I \}$, 
 all having stable rank one, all possessing  the same Elliott invariant, and
$\rc (\Cu (A_\gamma), [1_{A_\gamma}])=r$ for all $\gamma \in I$. 
Using  Theorem~\ref{ThmManufactureRc}, we get
$\rc \left( 
\Cu (A_\gamma), \ \Fu (\Cu(A_\gamma)  \right) =[0, \infty)$ all $\gamma \in I$. 
\end{proof}
We refer to \cite{Vil99} for 
Villadsen algebras and to 
Theorem~6.1 of \cite{ELN24} the classification of a certain type
Villadsen algebras by both the Elliott invariant and the radius of comparison.
The above proposition, which is strengthened version of Theorem~3.11 of \cite{HP25}, in fact shows that the combination of the Elliott invariant and the surjective 
of the radius of comparison function cannot classify simple AH algebras in general.
%
\section{Examples}\label{SectionExam}
In this section, we present a variety of examples that satisfy our theorems and  highlight some important and meaningful applications of theorems. 

Here is an example that
the value of 
the rank ratio function at $(y, x)$
 is not equal to the reciprocal of its value at $(x, y)$ for $x, y \in \Fu (S)$ with $x \leq y$.
\begin{exa}\label{Exa.RecipricalValue}
Let $A$ be a unital residually stably finite exact C*-algebra.
Let $B= M_2 (A) \oplus M_2 (A)$.
Set
 \[
 a = \left( \begin{pmatrix}
1 & 0\\
0 & 0
\end{pmatrix},
1_{M_2 (A)}
\right)
\quad
\mbox{and}
\quad
b=
 \left(1_{M_2 (A)},
1_{M_2 (A)}
\right).
\]
For, any $\tau \in \T^1(B)$, we have
$d_{\tau} ([b])=1$ and 
\[
d_{\tau} ([a]) =  \tau (a)
= t \cdot \tau_1 \left( \begin{pmatrix}
1 & 0\\
0 & 0
\end{pmatrix}\right) + 
(1-t) \cdot \tau_2 (1_{M_2 (A)})= 1- \frac{t}{2},
\]
where $t\in [0, 1]$ and $\tau_1, \tau_2 \in \T^1(M_2 (A))$. 
Using this and Proposition~\ref{Prp.rhoequiavalent}(\ref{Prp.rhoequiavalent.a})
\[
\rho ([a], [b]) = 
\sup_{\tau \in \T^1(A)} \left( \frac{d_{\tau} ([a])}{d_{\tau} ([b])} \right)=
 1
\quad
\mbox{and}
\quad
\rho ( [b], [a]) = \sup_{\tau \in \T^1(A) } \left( \frac{d_{\tau} ([b])}{d_{\tau} ([a])} \right)=2.
\]
\end{exa}
Here is an example that both the rank ratio function at any given pair of full elements and  the relative radii of comparison associated to them are zero. 
\begin{exa}\label{Exa.PurelyInfinite}
Let $A$ be a  purely infinite C*-algebra
and let   $a, b \in A\otimes \cK$ be full.
Then:
\begin{enumerate}
\item\label{Exa.PurelyInfinite.a}
It follows from Proposition~3.5 of \cite{KR02} that $A \otimes \cK$ is  also purely infinite. 
\item\label{Exa.PurelyInfinite.b}
By Proposition 5.1 of
\cite{KR00}, we have $\cF_{[a]} (\Cu (A))=\cF_{[b]} (\Cu (A))=\emptyset$. Then,
 by  Lemma~\ref{Rkiszero}, we have 
  $\rho\big([b], [a]\big)=\rho\big([a], [b]\big)=0$.
  \item\label{Exa.PurelyInfinite.c}
  It follows from Lemma~\ref{RcPurelyInfinite}  that 
  $
  \rc (\Cu (A), [a])= \rc (\Cu (A), [b])=0$.
  \item\label{Exa.PurelyInfinite.d}
  By (\ref{Exa.PurelyInfinite.c}) and Definition~\ref{Rciprical}, the range of 
  $\rc (\Cu(A), \cdot )$ is $\{0\}$ and the range of
  $\Irc_{\Cu (A)}$ is $\{\infty\}$. 
  \end{enumerate}
\end{exa}
Here we give an example that the rank ratio function is equal to one. 
\begin{exa}\label{rc.ofCX}
Let $X$ be a compact metric space and let 
$a \in \Fu (C(X))$ be positive. 
Set $\dt = \inf_{x\in X} a(x)$. 
\begin{enumerate}
\item
Since $a$ is full, it follows from Lemma~\ref{FullElements_Lem} that 
$a(x)>0$ for all $x \in X$. So $\dt>0$.
\item
Choose $n \in \N$ such that $\frac{1}{n} < \dt$. So, for every $m \geq n$, we have 
$(a-\tfrac{1}{m})_+ \sim a \sim 1_A$. 
Using this and Proposition~(\ref{Prp.rhoequiavalent})(\ref{Prp.rhoequiavalent.c}), we get $\lim_{m \to \infty} \rho ([a], [(a-\tfrac{1}{m})_+]) =1$. 
\end{enumerate}
\end{exa}

Here is an example where $\rho (y,x)$ is $0$ or
$\infty$ 
for $x, y \in \Fu(S)$. 
\begin{exa}\label{rankration0Infty}
Let $S$ be a $\Cu$-semigroup and let $x\in S\setminus \{0\}$.
We have
\[
\lambda (\infty \cdot x) =
\begin{cases} 
0, & \text{if } \lambda(x) = 0, \\
\infty, & \text{if } \lambda (x) \geq 0.
\end{cases}
\]
\begin{enumerate}
\item
For every $r \in (0, \infty)$, we have 
$\lambda (x) \leq r \cdot \lambda (\infty \cdot x)$. So, $\rho(x, \infty \cdot x)=0$. 
\item
We also have
$
\rho( \infty \cdot x, x)=
\begin{cases}
        0, & \text{if }  \cF_x (S) = \emptyset\\
        \infty, & \text{if }\cF_x (S) \neq \emptyset.
    \end{cases}
    $
\end{enumerate}
\end{exa}

Here is an example satisfying Theorem~\ref{LimofRank}(\ref{LimofRank.b}). 
\begin{exa}\label{ExprhoofNA}
Let $S$ be a $\Cu$-semigroup and let $a\in \Full(A)$. 
 It follows from Theorem~\ref{LimofRank}(\ref{LimofRank.b})
that 
\[
\rho (\infty \cdot a, a) \geq \rho (\infty \cdot a, 2 \cdot a) \geq \ldots \geq \rho (\infty \cdot a, n \cdot a) \geq \rho (\infty \cdot a, (n+1) \cdot a) \geq \ldots \geq \rho (\infty \cdot a, \infty \cdot a).
\]
Using Lemma~\ref{IncDesrho}(\ref{IncDesrho.a}), we get 
\[
\rho (\infty \cdot a, a) \geq \frac{1}{2} \cdot \rho (\infty \cdot a, a) \geq \ldots \geq  \frac{1}{n} \cdot \rho (\infty \cdot a, a) \geq  \frac{1}{n+1} \cdot \rho (\infty \cdot a, a) \geq \ldots \geq \rho (\infty \cdot a, \infty \cdot a).
\]
\begin{enumerate}
\item\label{ExprhoofNA.a}
If $\cF_a(S) \neq \emptyset$,
then, for every $n \in \N$, we have
\[
(\rho(\infty \cdot a, na))_{n=1}^{\infty}=\infty
\quad
\mbox{ and } 
\quad
\rho(\infty \cdot a, \infty \cdot a )=0.
\]
\item\label{ExprhoofNA.b}
If $\cF_a(S) =\emptyset$,
then, for every $n \in \N$, we have
\[
\rho(\infty \cdot a, n\cdot a))_{n=1}^{\infty}=
\rho(\infty \cdot a, \infty \cdot a )=0.
\]
\end{enumerate}
\end{exa}
Here is an example of nonunital C*-algebra satisfying Theorem~\ref{Rkrc}. 
\begin{exa}\label{nonunitalrcrankratio}
Let $A$ be a (not necessarily unital) C*-algebra.  Suppose that there exists $a \in A_+$ and $\ep>0$ such that
 $(a-\ep)_+$ is full. This can occurs for example when  the  primitive ideal space of $A$ is compact (see Proposition~3.1 of  \cite{TT15}).
 Then:
\begin{enumerate}
\item 
It follows from Proposition~\ref{rcProperties}(\ref{rcProperties.a_1})
that $\rc (\Cu (A), [a]) \leq \rc (\Cu (A), [(a-\ep)])$. 
\item
If $\cF_{[a]} (\Cu(A)) \neq \emptyset$ and 
$\rho([a], [(a- \ep)_+])< \infty$, then, by Theorem~\ref{Rkrc}(\ref{Rkrc.b}), 
\begin{align*}
\frac{1}{\rho([(a-\ep)_+], [a])} \cdot 
\rc (\Cu(A), [a]) & \leq \rc (\Cu (A), [(a- \ep)_+]) 
\\& \leq 
\rho([a], [(a- \ep)_+]) \cdot \rc (\Cu(A), [a]).
\end{align*}
\end{enumerate}
\end{exa}
Here is an example of unital C*-algebra satisfying Theorem~\ref{Rkrc}. 
\begin{exa}\label{RCPand1}
Let $A$ be a unital residually stably finite C*-algebra and let $p \in \cup_{n=1}^{\infty} M_n (A)_+$ be a  full projection.
Then:
\begin{enumerate}
\item\label{RCPand1.a}
It follows from Theorem~\ref{CorRrc=rc}
that
$\Rrc(A, a) = \rc (\Cu (A), [a])$. 
\item\label{RCPand1.b}
Using Proposition~\ref{Prp.rhoequiavalent}(\ref{Prp.rhoequiavalent.c}), we get 
\[
\rho ([1_A], [p])=  \frac{1}{\inf_{\tau \in \QT^1_2(A)}\tau(p)}
\ 
\mbox{ and }
\ 
\rho ([p], [1_A])= \sup_{\tau \in \QT^1_2(A)}\tau(p).
\]
\item\label{RCPand1.c}
Set 
\[
\kappa=
\frac{1}{\sup_{\tau \in \QT_2^1(A)}(\tau(p))}
\quad
\mbox{ and }
\quad
\xi= \frac{1}{\inf_{\tau \in \QT_2^1(A)}(\tau(p))}.
\]
It follows from Lemma~\ref{RkRatioProperty}(\ref{RkRatioProperty.f}), Theorem~\ref{Rkrc}(\ref{Rkrc.b}), and (\ref{RCPand1.a}) that
\[
\frac{1}{\kappa} \cdot  \rc(\Cu (A), [1_A])
\leq
 \rc (\Cu (A), [p]) 
 \leq 
 \frac{1}{\xi} \cdot
 \rc(\Cu(A), [1_A]).
\]
\end{enumerate}
\end{exa}
%
%
The following example provide two things. 
The fist one is providing a sequence $(b_m)_{m \in \N}$ of full elements  such that 
the radius of comparison relative to each $b_m$ exists and is finite, but 
the radius of comparison relative to $\lim_{m \to \infty} b_m$ is not even well-defined. 
The second part is providing 
an example 
satisfying 
Theorem~\ref{MainThmForRCLim}. 
\begin{exa}\label{EXlimiISZero}
Let $X$ be a compact metric space with $\dim X = 2k+1$ for $k \in \Ntwo$ and let $n \in \Ntwo$. 
Let $A=C(X, M_n)$ and 
let 
$(e_{j,k})^n_{j,k=1}$ denote the standard system of matrix units in
$M_n(\mathbb{C})$. 
For all $m \in \N$, define $a, a_m, b_m \colon X \to M_n$  by
\[
a_m (x)= 
e_{n, n} \otimes 1+ \sum_{j=1}^{n-1} e_{j,j} \otimes \frac{1}{m}
, \quad
a (x)= 
e_{n, n} \otimes 1, 
\quad
\mbox{and}
\quad
b_m (x)=
 \sum_{j=1}^{n} e_{j,j} \otimes \frac{1}{m}. 
\]
Then:
\begin{enumerate}
\item \label{EXlimiISZero.a}
$\lim_{m \to \infty} a_m = e_{n, n} \otimes 1$
and 
$\lim_{m \to \infty} b_m= 0_{n \times n}$ 
\item\label{EXlimiISZero.b}
It is obvious that $b_m \sim 1_{A}$. Using 
Proposition~\ref{rcProperties}(\ref{rcProperties.b}) at the first step, 
Proposition~\ref{rcProperties}(\ref{rcProperties.d})
at the second step, 
Theorem~\ref{CorRrc=rc} and  the fact that $A$ is residually stably finite at the third step, and 
Corollary~1.2(1) of \cite{EN13} at the last step,
 we get, for all $m \in \N$, 
 \begin{align*}
 \rc (\Cu(A), [b_m])
 = 
 \rc (\Cu (C(X)), [1_A]) 
 &=
\frac{1}{n} \cdot \rc (\Cu (C(X)), [1_{C(X)}]) 
\\&= 
\frac{1}{n} \cdot
\Rrc (C(X), 1_{C(X)})= \frac{k-1}{n}.
\end{align*}
However, $\rc (\Cu (A), \lim_{m \to \infty} b_m)$ is not even well-defined as 
$\lim_{m \to \infty} b_m$ is not  full by (\ref{EXlimiISZero.a}). 
\item\label{EXlimiISZero.c}
Using Proposition~\ref{Prp.rhoequiavalent}(\ref{Prp.rhoequiavalent.c}), we get 
$\rho([a], [a_m])= \frac{1}{n}$
and 
$\rho([a_m], [a])= n$. 
\item 
Using (\ref{EXlimiISZero.c}) at the first step, using Proposition~\ref{rcProperties}(\ref{rcProperties.b}) at the second  step, and
using $[a]= [1_{C(X)}]$ and Proposition~(\ref{rcProperties})(\ref{rcProperties.b}) at the third step, we get
\begin{align*}
 \rho([a_m], [a]) \cdot \rc \left(\Cu (A), [a_m] \right)
 &=
 n \cdot \rc \left(\Cu (A), n \cdot [1_{C(X)}] \right) 
\\ &=
\rc (\Cu (A), [1_{C(X)}] ) 
\\&= \rc (\Cu (C(X)), [a]).
\end{align*}
\end{enumerate}
The set of all Radon Borel probability measures on $X$ is denoted by $\cM_1 (X)$. 
Here is an example satisfying Theorem~\ref{MainThmForRCLim} and we also
 concretely compute the rank ratio function at a given point.  This gives us a better understanding of the rank ratio function 
 in terms of measure theory. 
\begin{exa}\label{RankratioMeasureTheo}
Let $n, t \in \Ntwo$ and let $X=[0, \tfrac{1}{n}]^t$. Let $A=M_n(C(X))$ and let 
 $f \in C(X)_+$ be given by
$f((x_1, x_2, \ldots, x_t))= \sum_{j=1}^{t} x_j$. 
Let 
$(e_{j,k})^n_{j,k=1}$ denote the standard system of matrix units in
$M_n(\mathbb{C})$. 
For each $m \in \N$, set the elements $a, a_m \in M_n \otimes C(X)_+$ by
\[
a_m= e_{1, 1} \otimes \left(f-\tfrac{1}{m}\right)_+  + \sum_{j=2}^{n} e_{j,j} \otimes 1_{C(X)}
\quad
\mbox{and}
\quad
a=
e_{1, 1} \otimes f + \sum_{j=2}^{n} e_{j,j} \otimes 1_{C(X)}.
\]
\begin{enumerate}
\item\label{RankratioMeasureTheo.a}
By Lemma~\ref{FullElements_Lem}, it is clear that $a$ and $a_m$ for all $m \in \N$ are full. 
\item\label{RankratioMeasureTheo.b}
Let $\tau \in \T^1(A)$. So, by Riesz Representation Theorem, 
there exists $\mu \in \cM_1 (X)$ such that 
\[
\tau \left((g_{j, k})_{j, k=1}^n \right) =
\frac{1}{n} \cdot \sum_{k=1}^{n} \int_{X} g_{k, k} (x) \ \mathrm{d} \mu(x),
\]
where $g_{j, k} \in C(X)$ for $1 \leq j, k\leq n$.
For $m \in \Ntwo$, 
set $Y_m= \{ x \in X \colon f(x) \leq \frac{1}{m}\}$. For every $m \in \Ntwo$, we get, using Monotone Convergence Theorem at the second step, 
\begin{align}\label{Eq1.2025.04.10}
d_{\tau} (a_m) = \lim_{s \to \infty} \tau (a_m^{1/s}) &=
\frac{n-1}{n} + \frac{1}{n} \cdot 
\int_{X \setminus Y}  \lim_{s \to \infty} \left((f- \tfrac{1}{m})_+ (x)\right)^{1/s} \ \mathrm{d} \mu(x)
\\\notag&=
\frac{n-1}{n}  + \frac{1}{n} \cdot \mu (X \setminus Y)= 1- \frac{1}{n}  \cdot \mu (Y).
\end{align}
Similarly, we have
\begin{equation}\label{Eq2.2025.04.10}
d_{\tau} (a) =  1- \frac{1}{n} \cdot \mu (\{0\}).
\end{equation}
\item\label{RankratioMeasureTheo.c}
Using Proposition~\ref{Prp.rhoequiavalent}(\ref{Prp.rhoequiavalent.c}) at the first step and considering the
the Dirac measure at the point $(\tfrac{1}{tm}, \ldots, \tfrac{1}{tm} ) \in X$ at the last step, we get 
\[
\rho([a], [a_m]) =\sup_{\tau \in \T^1(A)} \left( \frac{d_{\tau} (a)}{d_{\tau} (a_m)} \right)
\overset{\text{(\ref{Eq1.2025.04.10})}}{\underset{\text{(\ref{Eq2.2025.04.10})}}{=}}
\sup_{\mu \in \cM_1 (X)} \left( \frac{ 1- \frac{1}{n} \cdot \mu (\{0\})}{1- \frac{1}{n}  \cdot \mu (Y)} \right) = \frac{n}{n-1}.
\]
Using Proposition~(\ref{Prp.rhoequiavalent})(\ref{Prp.rhoequiavalent.c}) at the first step and considering the
the Dirac measure at the point $(\tfrac{1}{m}, \ldots, \tfrac{1}{m} ) \in X$ at the last step, we get 
\[
\rho([a_m], [a]) =\sup_{\tau \in \T^1(A)} \left( \frac{d_{\tau} (a_m)}{d_{\tau} (a)} \right) 
\overset{\text{(\ref{Eq1.2025.04.10})}}{\underset{\text{(\ref{Eq2.2025.04.10})}}{=}}
\sup_{\mu \in \cM_1 (X)} \left( \frac{1- \frac{1}{n}  \cdot \mu (Y)}{ 1- \frac{1}{n} \cdot \mu (\{0\})} \right) = 1.
\]
\item\label{RankratioMeasureTheo.d}
By Theorem~\ref{MainThmForRCLim}, we have
\[
\rc (\Cu(A), [a]) \leq \lim_{m \to \infty} \rc (\Cu (A), [a_m]) \leq 
\frac{n}{n-1} \cdot \rc (\Cu(A), [a]).
\]
The value of $\rc (\Cu(A), [a])$ depends on $t$, i.e., the covering dimension of $X$.
\end{enumerate}
\end{exa}
\end{exa}

The following example is an algebraic satisfying Proposition~\ref{ExtremeInfiniteRho}. 
\begin{exa}\label{AlgebraciExample}
Let $S$ be a $\Cu$-semigroup. Let $x \in \Full(S)$. Let $\cF_x (S) \neq \emptyset$. Then:
\begin{enumerate}
\item
By Example~\ref{ExprhoofNA}(\ref{ExprhoofNA.a}), we have 
$\rho(\infty \cdot x, n \cdot x)= \infty$ for all $n\in \N$ and
$\rho(\infty \cdot x, \infty \cdot x)= 0$
\item
By Proposition~\ref{rcProperties}(\ref{rcProperties.a_1}) that 
\[
 \rc (S, x) \geq \rc (S, 2\cdot x) \geq 
 \ldots \geq \rc (S, n\cdot x) \geq \rc (S, (n+1)\cdot x)
 \ldots \geq \rc (S, \infty \cdot x).
 \]
 Then, by Proposition~\ref{rcProperties}(\ref{rcProperties.b}), we have 
 \[
 \rc (S, x) \geq \frac{1}{2} \cdot \rc (S, x) \geq 
 \ldots \geq \frac{1}{n} \cdot  \rc (S, x) \geq \frac{1}{n+1} \cdot \rc (S, x) \geq 
 \ldots \geq \rc (S, \infty \cdot x).
 \]
Assume that $\rc(S, x) <\infty$. Therefore, using Proposition~\ref{rcProperties}(\ref{rcProperties.c}) at the last step,  
\[
\lim_{n \to \infty} \rc(S, n.x) = 
\lim_{n \to \infty} \frac{1}{n} \cdot \rc(S, x) = 0=\rc(S, \infty \cdot x).
\] 
\end{enumerate}
\end{exa}

\end{document}